\documentclass[11pt,letterpaper]{amsart} 

\usepackage[english]{babel}
\usepackage{eurosym}
\usepackage{amsfonts}
\usepackage{amsmath}
\usepackage{amsthm}
\usepackage{amssymb}
\usepackage{mathrsfs}
\usepackage{tensor} 
\usepackage{stmaryrd} 
\usepackage{amscd}

\usepackage{tikz}

\usepackage{framed}
\usepackage{color}
\definecolor{shadecolor}{gray}{0.875}
\definecolor{col}{RGB}{42, 95, 151}
\usepackage[colorlinks=true, linkcolor=col, citecolor=col, filecolor = col, menucolor = col, urlcolor = black]{hyperref}

\usepackage{enumerate}

\theoremstyle{plain}
\newtheorem{theo}{Theorem}[section]
\newtheorem{lemm}[theo]{Lemma}
\newtheorem{prop}[theo]{Proposition}

\newtheorem{coro}[theo]{Corollary}
\theoremstyle{definition}
\newtheorem{rema}[theo]{Remark}
\newtheorem{defi}[theo]{Definition}
\newtheorem{exam}[theo]{Example}

\newcommand{\llbr}{\ensuremath{\llbracket}}
\newcommand{\rrbr}{\ensuremath{\rrbracket}}

\newcommand{\llpar}{(\negthinspace(}
\newcommand{\rrpar}{)\negthinspace)}

\DeclareMathOperator{\Spec}{Spec}

\newcommand{\Hom}{\ensuremath{\mathrm{Hom}}}

\newcommand{\LL}{\mathbb{L}}
\newcommand{\N}{\mathbb{N}}
\newcommand{\Z}{\mathbb{Z}}
\newcommand{\Q}{\mathbb{Q}}
\newcommand{\R}{\mathbb{R}}

\newcommand{\A}{\mathbb{A}}

\newcommand{\PP}{\mathbb{P}}
\newcommand{\RR}{\mathbb{R}}

\newcommand{\cU}{\mathscr{U}}
\newcommand{\cV}{\mathscr{V}}
\newcommand{\cX}{\mathscr{X}}
\newcommand{\cY}{\mathscr{Y}}
\newcommand{\cZ}{\mathscr{Z}}

\newcommand{\Var}{\ensuremath\mathrm{Var}}

\newcommand{\Supp}{\ensuremath\mathrm{Supp}}
\newcommand{\ord}{\ensuremath\mathrm{ord}}
\newcommand{\Gro}{\ensuremath\mathbf{K}}
\newcommand{\SB}{\ensuremath\mathrm{SB}}
\newcommand{\sbmap}{\ensuremath\mathrm{sb}}

\newcommand{\trop}{\ensuremath\mathrm{trop}}

\newcommand{\sbir}[1]{[#1]_{\mathrm{sb}}}

\DeclareMathOperator{\codim}{codim}

\DeclareMathOperator{\Vol}{Vol}
\DeclareMathOperator{\Volsb}{Vol_{\mathrm{sb}}}

\numberwithin{equation}{subsection}

\definecolor{mycol}{rgb}{0.36, 0.54, 0.66}
\definecolor{mycol}{rgb}{0.0, 0.5, 1.0}
\definecolor{mycol}{rgb}{0.13, 0.67, 0.8}
\definecolor{mycol}{rgb}{0.19, 0.55, 0.91}
\definecolor{mycol}{RGB}{62, 118, 201}
\definecolor{mycol2}{rgb}{0.0, 0.5, 0.0}

\usepackage{tikz-3dplot}

\usepackage{breqn}

\tdplotsetmaincoords{70}{120}

\title{Tropical degenerations and stable rationality}

\author{Johannes Nicaise}
\address{Imperial College,
Department of Mathematics, South Kensington Campus,
London SW7 2AZ, UK, and KU Leuven, Department of Mathematics, Celestijnenlaan 200B, 3001 Heverlee, Belgium} \email{j.nicaise@imperial.ac.uk}

\author{John Christian Ottem}
\address{Department of Mathematics, University of Oslo, Box 1053, Blindern, 0316 Oslo, Norway}
\email{johnco@math.uio.no}
\date{\today}

\begin{document}
\begin{abstract}
We use the motivic obstruction to stable rationality introduced by Shinder and the first-named author to establish several new classes of stably irrational hypersurfaces and complete intersections. In particular, we show that very general quartic fivefolds and complete intersections of a quadric and a cubic in $\mathbb P^6$ are stably irrational. An important new ingredient is the use of tropical degeneration techniques. 
\end{abstract}

\maketitle
\thispagestyle{empty}
\vspace{-0.5cm}
{
  \hypersetup{linkcolor=black}
\tableofcontents 
}

\section{Introduction}
A central question in algebraic geometry is the {\em rationality problem}, which asks whether a given algebraic variety $X$ over a field $k$ is rational,  or, more generally, {\em stably rational}, which means that $X\times_k \mathbb P^m_k$ is rational for some $m\ge 0$.  Here the case of hypersurfaces in $\PP^{n+1}_k$ is particularly important, with classical results going back to Clemens--Griffiths \cite{CG} and Iskovskikh--Manin \cite{IM} in the early 1970s, and Koll\'ar \cite{kollar} in the 1990s. In recent years, there has been much progress on this problem, triggered by Voisin's specialization technique \cite{voisin}, as well as subsequent developments (see \cite{CTP, totaro, schreieder}). For instance, Schreieder proved in \cite{schreieder} that a very general projective hypersurface of degree $d$ and dimension $n\geq 3$ is stably irrational if $d\geq \log_2n + 2$, generalizing existing results by Totaro \cite{totaro} and Kollár \cite{kollar}. Despite these breakthroughs, many fundamental questions remain open. Famous unsolved problems include the rationality of a very general cubic fourfold, the stable rationality of a very general cubic threefold, and the rationality of smooth quartic hypersurfaces. 
 

In \cite{NiSh}, Evgeny Shinder and the first-named author introduced a powerful new tool in the study of rationality properties of algebraic varieties over fields of characteristic zero. They showed that, to any strictly toroidal one-parameter degeneration, one can attach a motivic obstruction to the stable rationality of the geometric generic fiber. They furthermore used this obstruction to prove that the locus of stably rational geometric fibers in mildly singular families is closed under specialization. This method was further improved in \cite{KT} to prove specialization of rationality instead of stable rationality. A unified treatment of these results was given in \cite{NO}.

 The aim of the present paper is to apply these techniques to concrete rationality questions. Our main applications are the following (all in characteristic zero):\begin{itemize}
\item A very general quartic fivefold is stably irrational (Corollary \ref{cor:quartic-fivefold}). Stable irrationality of quartic hypersurfaces was previously known only in dimensions $n\leq 4$ \cite{CTP, totaro}. In dimensions $n> 5$, we prove the stable irrationality of very general hypersurfaces of degree $d\geq 4$ assuming stable irrationality of a special quartic in dimension $n-1$ (Theorems \ref{theo:quartic1} and \ref{theo:quartic2}). We also establish analogous conditional results in higher degrees (Proposition \ref{doublehypersurface}).

\item Bounds for the stable irrationality of hypersurfaces in products of projective spaces. In particular, we classify the bidegrees corresponding to stably irrational hypersurfaces of dimension at most $4$ in products of projective spaces  (Proposition \ref{prop:productproj} and Theorem \ref{P1P4}). 

\item If a very general hypersurface of degree $d$ and dimension $n$ is stably irrational, 
then a very general hypersurface of degree $d+1$ and dimension $n$ or $n+1$ is also stably irrational (Theorem \ref{theo:slope}). 

\item A very general intersection of a quadric and a cubic in $\mathbb{P}^6_k$ is stably irrational (Theorem \ref{theo:quadcub}). Apart from the cubic fourfold, this was the only open case for  complete intersection fourfolds. We also prove that a very general intersection of two cubics in $\mathbb{P}^7_k$ is stably irrational (Corollary \ref{coro:33}).

\item Significantly sharper bounds for stable irrationality of very general complete intersections, at the level of Schreieder's bounds for hypersurfaces \cite{schreieder}, see Theorems \ref{theo:ci} and \ref{theo:ci2}. We use Schreieder's bounds for hypersurfaces as an essential ingredient in the proof. As a special case of Theorem \ref{theo:ci2}, we obtain that, when $n\geq 4$ and $r\geq n-1$, a very general complete intersection of $r$ quadrics in $\mathbb{P}^{n+r}_k$ is stably irrational (Corollary \ref{coro:quadrics}). This generalizes results by Hassett, Pirutka \& Tschinkel \cite{HPT3quad} and Chatzistamatiou \& Levine \cite{CL}. 
\end{itemize}

Given a strictly toroidal one-parameter degeneration, 
 the motivic obstruction to stable rationality of the geometric generic fiber takes the form of an alternating sum of the virtual stable birational types of the strata in the degenerate fiber. In this way, the stable irrationality of the geometric generic fiber can be deduced from the stable irrationality of special varieties of lower dimensions and/or degrees. For instance, we degenerate the quartic fivefold to a union of two varieties which intersect along a very general quartic double fourfold, which is known to be stably irrational \cite{HPTdouble}.
 To prove the non-triviality of the obstruction, it is essential to control the stable birational types of different strata in the special fiber to make sure that the stably irrational strata do not cancel out in the alternating sum. 
 An important tool in this context is the result by Shinder on variation of stable birational types in families \cite{shinder}. We refine and extend this result in a more general setting (Theorem \ref{theo:variation}) and use it at various places to prove the non-triviality of the motivic obstruction to stable rationality.

 In order to apply these methods to concrete problems, one needs to construct suitable degenerations. For this purpose, we use the theory of tropical compactifications, which associates toroidal degenerations to a large class of polyhedral subdivisions of the Newton polytope of a hypersurface in an algebraic torus. We compute the motivic obstruction for such degenerations by means of the tropical formulas for the motivic nearby fiber in \cite{NPS}; see Theorems \ref{theo:nondeg} and \ref{theo:main}. This makes it possible to apply combinatorial methods to rationality problems in algebraic geometry: one tries to prove stable irrationality of a hypersurface by showing that its Newton polytope $\Delta$ contains the Newton polytope of a hypersurface that is already known to be stably irrational, and such that this smaller polytope is not contained in the boundary of $\Delta$.

Degeneration techniques have been widely used in the study of rationality questions, especially after Voisin's pioneering work on the specialization technique \cite{voisin, CTP, totaro, schreieder}. A distinctive feature of our approach is that our degenerate fibers typically have many irreducible components, whereas Voisin considered degenerations with integral fibers\footnote{However, applications involving reducible degenerate fibers have been worked out in \cite{totaro,CL,CT}. We are grateful to Jean-Louis Colliot-Th\'el\`ene for pointing out these references.}. Moreover, in many of our applications, the stable irrationality of the geometric generic fiber is deduced from the stable irrationality of strata in the special fiber of smaller dimension. To the best of our knowledge, such a systematic reduction in dimension has not been achieved with other methods, and we are confident that it will have further applications beyond those explored in the present article.

The paper is organized as follows. In Section \ref{sec:sbir} we recall the main results on the motivic obstruction to stable rationality from \cite{NiSh} and \cite{NO}, and phrase them in a form that is suitable for our applications. Section \ref{sec:trop} is the technical heart of the paper: we introduce the necessary tools from tropical geometry, and give a tropical formula for the motivic obstruction (Theorems \ref{theo:nondeg} and \ref{theo:main}). In the remaining sections, we discuss the applications that were outlined above: we consider hypersurfaces in projective spaces in  Section \ref{sec:app}, hypersurfaces in products of projective spaces in Section \ref{sec:prodproj}, and finally complete intersections in projective spaces in Section \ref{sec:ci}.

\medskip

\noindent {\bf Acknowledgements.} We are grateful to Jean-Louis Colliot-Th\'el\`ene, Daniel Litt, Sam Payne, Alex Perry,  Stefan Schreieder, Kristin Shaw, Evgeny Shinder and Claire Voisin for helpful discussions at various stages of this project. 
 This project was started while both authors were visiting the Higher School of Economics in Moscow  for a workshop on birational geometry. They would like to thank the HSE and the organizers of the workshop, Yuri Prokhorov and Konstantin Shramov, for their hospitality. JN is supported by EPSRC grant EP/S025839/1, grants G079218N and G0B17121N of the Fund for Scientific Research--Flanders, and long term structural funding (Methusalem
grant) of the Flemish Government. JCO is supported by the Research Council of Norway project no. 250104.

\medskip

\noindent {\bf Notations.} We denote by $k$ an algebraically closed field of characteristic zero, and by
$K$ the field of Puiseux series over $k$, that is,
$$K=\bigcup_{n>0}k\llpar t^{1/n}\rrpar.$$
We consider the $t$-adic valuation
$$\mathrm{ord}_t\colon K^{\times}\to \Q,\,a\mapsto \mathrm{ord}_t(a)$$
on $K$. For every positive integer $n$ and every non-zero element $$a=\sum_{i\in \Z}a_i t^{i/n}$$ of $k\llpar t^{1/n}\rrpar$, the value $\mathrm{ord}_t(a)$ is the minimum of the set
$\{i/n\,|\,i\in \Z,\,a_i\neq 0\}.$
 We denote by $R$ the valuation ring in $K$; thus
 $$
 R=\{0\}\cup \{a\in K^{\times}\,|\,\ord_t(a)\geq 0\}=\bigcup_{n>0}k\llbr t^{1/n}\rrbr.$$

We adopt the following convention in our notations for projective bundles: when $S$ is a scheme and $\mathcal{E}$ is a locally free sheaf of finite rank on $S$, we denote by $\mathbb{P}\mathcal{E}$ the projective bundle
$\mathrm{Proj}\,\mathrm{Sym}_{\mathcal{O_S}}(\mathcal{E}^{\vee})$
 on $S$. In particular, when $V$ is a vector space over a field $F$, then $\mathbb{P}V$ is a projective space over $F$ whose $F$-points correspond to one-dimensional subspaces of $V$.

\section{Specialization of stable birational types}\label{sec:sbir}
In this section, we recall some results and tools from \cite{NiSh} and \cite{NO} that we will need in this paper.

\subsection{Stable birational types and the theorem of Larsen \& Lunts}
Let $F$ be a field. Two  $F$-schemes of finite type $X$ and $Y$ are called {\em stably birational} if $X\times_F \mathbb{P}_F^{\ell}$ is birational to $Y\times_F \mathbb{P}_F^m$ for some non-negative integers $\ell$ and $m$. We say that $X$ is {\em stably rational} if it is stably birational to the point $\Spec F$. We denote by $\SB_F$ the set of stable birational equivalence classes of integral\footnote{We follow the convention that integral schemes are assumed to be non-empty.} $F$-schemes of finite type; the equivalence class of $X$ will be denoted by $\sbir{X}$.
 We write $\Z[\SB_F]$ for the free abelian group on the set $\SB_F$.
 For every $F$-scheme $Z$ of finite type, we set
 $$\sbir{Z}=\sum_{i=1}^r\sbir{Z_i}$$ in $\Z[\SB_F]$,   where $Z_1,\ldots,Z_r$ are the irreducible components of $Z$ (with their induced reduced structures). In particular, $\sbir{\emptyset}=0$.  Note that two reduced $F$-schemes of finite type are stably birational if and only if they define the same element in $\Z[\SB_F]$.

 There exists a unique ring structure on $\Z[\SB_F]$ such that
 $$\sbir{X}\cdot \sbir{Y}=\sbir{X\times_F Y}$$ for all $F$-schemes $X$ and $Y$ of finite type.
 The identity element for the ring multiplication is $\sbir{\Spec F}$, the class of stably rational $F$-schemes of finite type.

 We denote by $\Gro(\Var_F)$ the Grothendieck ring of $F$-varieties.
 As an abelian group, it is characterized by the following presentation:
\begin{itemize}
\item {\em Generators:} isomorphism classes $[X]$ of $F$-schemes $X$ of finite type.
\item {\em Relations:} whenever $Y$ is a closed subscheme of $X$, we have
$$[X]=[Y]+[X\setminus Y].$$
  \end{itemize}
 This abelian group has a unique ring structure such that $[X]\cdot [Y]=[X\times_F Y]$ for all
 $F$-schemes $X$ and $Y$ of finite type.
 We write $\LL=[\A^1_F]$ for the class of the affine line in $\Gro(\Var_F)$.
  To simplify the notation, if $\varphi\colon \Gro(\Var_F)\to A$ is any map to some set $A$,
then we will write $\varphi(Y)$ instead of $\varphi([Y])$ whenever $Y$ is an $F$-scheme of finite type.

 The following fundamental theorem was proved in \cite{LaLu}.
 \begin{theo}[Larsen \& Lunts]\label{theo:lalu}
  Assume that $F$ has characteristic zero.
There exists a unique ring morphism
$$\sbmap\colon \Gro(\Var_F)\to \Z[\SB_F]$$
that maps $[X]$ to $\sbir{X}$ for every smooth and proper $F$-scheme $X$.
It factors through an isomorphism
$$\Gro(\Var_F)/(\LL)\to \Z[\SB_F].$$
\end{theo}
 Beware that one usually has $\sbmap(Y)\neq \sbir{Y}$ when $Y$ is not smooth and proper; for instance, $\sbmap(\A^1_F)=\sbmap(\mathbb{P}^1_F)-\sbmap(\Spec F)=0$.

\subsection{The stable birational volume}
 The stable birational volume is a version of the nearby cycles functor for stable birational types that controls how these degenerate in families.
  It can be constructed in terms of a particular class of models.
 A monoid $M$ is called {\em toric} if it is finitely generated, integral, saturated, and sharp (that is, $0$ is the only invertible element in $M$). This is equivalent to saying that $M$ is isomorphic to the monoid of lattice points in a strictly convex rational polyhdral cone. To any monoid $M$ we can attach its monoid $R$-algebra $R[M]$; the monomial associated with an element $m\in M$ will be denoted by $x^m$.
 
Let $\cX$ be a flat separated $R$-scheme of finite presentation.
 We say that $\cX$ is {\em strictly toroidal} if, Zariski-locally on $\cX$, we can find a smooth morphism
 $$\cX\to \Spec R[M]/(x^m - t^q)$$ for some toric monoid $M$, some positive rational number $q$, and some element $m$ in $M$ such that $k[M]/(x^m)$ is reduced.
 
 If $\cX$ is strictly toroidal, then a {\em stratum} of the special fiber $\cX_k$ is a connected component $E$ of an intersection of irreducible components of $\cX_k$. The codimension $\codim(E)$ of the stratum $E$ is defined to be the codimension in $\cX_k$. 
 We will denote the set of strata of $\cX_k$ by $\mathcal{S}(\cX)$.

\begin{exam}\label{exam:standardmod}
Let $r$ and $s$ be positive integers, and let $a=(a_1,\ldots,a_r)$ and $b=(b_1,\ldots,b_s)$ be tuples of positive integers. With these data we associate the $R$-schemes 
\begin{eqnarray*}
\cX_{a} &=& \Spec R[x_{i,j}\,|\,i=1,\ldots,r;\,j=1,\ldots,a_i]/(t-\prod_{j=1}^{a_1}x_{1,j} ,\ldots,t-\prod_{j=1}^{a_r}x_{r,j}),
\\ \cY_{b} &=& \Spec R[y_{i,j}\,|\,i=1,\ldots,s;\,j=0,\ldots,b_i]/(ty_{1,0}-\prod_{j=1}^{b_1}y_{1,j} ,\ldots,ty_{s,0}-\prod_{j=1}^{b_s}y_{s,j}).
\end{eqnarray*}
Then it is easy to see that the schemes $\cX_a$, $\cY_b$ and $\cX_a\times_R \cY_b$ are strictly toroidal; we will only spell out the third case. Let $M$ be the quotient of the free monoid
$$\N e_{0}\oplus \bigoplus_{\substack{i=1,\ldots,r\\j=1,\ldots,a_i}}\N e_{i,j}\oplus \bigoplus _{\substack{i=1,\ldots,s\\j=0,\ldots,b_i}}\N f_{i,j}$$
by the congruence relations $$\left\{\begin{array}{ll} e_{0}=\sum_{j=1}^{a_i}e_{i,j}& \text{ for $i=1,\ldots,r$},
\\ e_{0}+f_{i,0}=\sum_{j=1}^{b_i}f_{i,j}& \text{ for $i=1,\ldots,s$}.
\end{array}\right.$$
Then $M$ is a toric monoid, and $$\cX_{a}\times_R \cY_b\cong \Spec R[M]/(x^{e_{0}}-t).$$
 It follows that every flat separated $R$-scheme $\cX$ of finite presentation that admits Zariski-locally a smooth morphism to a scheme of the form $\cX_a$, $\cY_b$ or $\cX_a\times_R \cY_b$ is also strictly toroidal. We say that $\cX$ is {\em strictly semi-stable} if it admits Zariski-locally a smooth morphism to a scheme of the form $\cX_a$ with $r=1$.

A similar calculation shows that the product of two strictly toroidal $R$-schemes is again strictly toroidal. This is one of the advantages of this class of schemes over the more restrictive class of strictly semi-stable $R$-schemes, which is not closed under products: every scheme of type $\cX_a$ is a product of strictly semi-stable schemes, but it is only strictly semi-stable if there is at most one index $i$ with $a_i\geq 2$. The easiest counterexample is $\Spec R[x,y,z,w]/(t-xy,t-zw)$, whose special fiber has four irreducible components of dimension $2$ intersecting at the origin, which never happens for strictly semi-stable schemes.
\end{exam}

Strictly toroidal degenerations arise naturally when we break up projective hypersurfaces into pieces of smaller degrees. The following construction is taken from Example 3.2.4 in \cite{NO}.
\begin{exam}\label{exam:toroidal}
 Let $f_0,\ldots,f_r\in k[z_0,\ldots,z_{n+1}]$ be general homogeneous polynomials of positive  degrees $d_0,\ldots,d_r$ such that $d_0=d_1+\ldots +d_r$. Then 
$$\cX=\mathrm{Proj}\,R[z_0,\ldots,z_{n+1}]/(tf_0-f_1\cdot \ldots \cdot f_r)$$ is strictly toroidal, but not strictly semi-stable at the points of $\cX_k$ where $f_0$ and at least two among $f_1,\ldots,f_r$ vanish.
\end{exam}

The following result from \cite{NO}, which generalizes \cite{NiSh}, furnishes a version of the nearby cycles functor for stable birational types.
\begin{theo}\label{theo:vol}
There exists a unique ring morphism
$$\Volsb\colon \Z[\SB_K]\to \Z[\SB_k]$$
such that, for every strictly toroidal proper $R$-scheme $\cX$ with smooth generic fiber $X=\cX_K$,
 we have
\begin{equation}\label{eq:volsb}
\Volsb(\sbir{X})=\sum_{E\in \mathcal{S}(X)}(-1)^{\codim(E)}\sbir{E}.
\end{equation}
\end{theo}
\begin{proof}
This is Corollary 3.3.5 in \cite{NO}.
\end{proof}

\begin{exam}\label{exam:goodmodel}
Let $X$ be a smooth and proper $K$-scheme, and assume that $X$ has a smooth and proper $R$-model $\cX$. Then $\Volsb(\sbir{X})=\sbir{\cX_k}$. In particular, $\Volsb(\sbir{\Spec K})=\sbir{\Spec k}$.
\end{exam}

\begin{coro}\label{coro:obstruction} \item
\begin{enumerate}
\item Let $X$ be a smooth and proper $K$-scheme. If 
$$\Volsb(\sbir{X})\neq \sbir{\Spec k}$$ in $\Z[\SB_k]$, then $X$ is not stably rational.

\item Let $\cX$ be a strictly toroidal proper $R$-scheme with smooth generic fiber $X=\cX_K$. If
$$ \sum_{E\in \mathcal{S}(\cX)}(-1)^{\codim(E)}\sbir{E}\neq \sbir{\Spec k}$$ in $\Z[\SB_k]$, then $X$ is not stably rational.
\end{enumerate}
\end{coro}
\begin{proof}
If $X$ is stably rational, then $\sbir{X}=\sbir{\Spec K}$ so that
$\Volsb(\sbir{X})=\sbir{\Spec k}$. The second part of the statement follows immediately from the expression for $\Volsb(\sbir{X})$  in terms of a strictly toroidal model in formula \eqref{eq:volsb}.
\end{proof}

We call the morphism  $\Volsb$ from Theorem \ref{theo:vol} the {\em stable birational volume}.
 By Corollary \ref{coro:obstruction}, we can use the stable birational volume as an obstruction to stable rationality.
 In \cite{NiSh}, Theorem \ref{theo:vol} was used to deduce the following important corollary, which we reproduce here for later use.

\begin{coro}[Theorem 4.1.4 and Corollary 4.1.5 in \cite{NiSh}]\label{coro:verygen}
Let $S$ be a Noetherian $\Q$-scheme, and let $X\to S$ and $Y\to S$ be smooth and proper morphisms. Then the set
$$\{s\in S\,|\,X\times_S \overline{s}\mbox{ is stably birational to }Y\times_S\overline{s} \mbox{ for any geometric point }\overline{s} \mbox{ based at }s\}$$
is a countable union of closed subsets of $S$.

In particular, the set
$$\{s\in S\,|\,X\times_S \overline{s}\mbox{ is stably rational, for any geometric point }\overline{s} \mbox{ based at }s\}$$
is a countable union of closed subsets of $S$.
\end{coro}

This result was further improved in \cite{KT} to get the analogous statements for birational equivalence and rationality instead of stable birational equivalence and stable rationality; see also \cite{NO} for a uniform treatment. For our current purposes, the stable version will be sufficient.

 The stable birational volume $\Volsb$ is closely related to the motivic volume that was constructed by Hrushovski and Kazhdan in \cite{HK}. This is a ring morphism 
$$\Vol\colon \Gro(\Var_K)\to \Gro(\Var_k)$$ that constitutes a motivic upgrade of the stable birational volume $\Volsb$ in the sense that we have a commutative diagram
$$\begin{CD}
\Gro(\Var_K) @>\Vol>> \Gro(\Var_k)
\\ @V\sbmap VV @VV\sbmap V
\\ \Z[\SB_K] @>\Volsb>> \Z[\SB_k].
\end{CD}$$ Thus every formula for the motivic volume can be specialized to a formula for the stable birational volume by applying Larsen and Lunts' realization morphism $\sbmap$ from Theorem \ref{theo:lalu}. We will apply this principle in the proof of Theorem \ref{theo:nondeg}.

\section{Tropical calculation of the stable birational volume}\label{sec:trop}
\subsection{Newton polytopes}
 We recall some standard definitions in tropical geometry; see \cite{MS} for background on Newton polytopes in tropical geometry, and \cite{CLS} for the necessary tools from toric geometry.
Let $n$ be an integer satisfying $n\geq -1$, and let $M$ be a free abelian group of rank $n+1$. We set $M_{\R}=M\otimes_{\Z}\R$.

 Let $F$ be a field, and let $$g=\sum_{m\in M}d_m x^m\in F[M]$$ be a non-zero Laurent polynomial with coefficients in $F$.
The {\em support} of $g$ is the set of lattice points $m$ in $M$ such that $d_m\neq 0$; we denote it by $\mathrm{Supp}(g)$. The convex hull of $\mathrm{Supp}(g)$ in the real vector space $M_{\R}$ is called the {\em Newton polytope} of $g$, and denoted by $\Delta_g$.

 Let $V$ be the real affine subspace of $M_{\R}$ spanned by $\Supp(g)$. We choose a point $m_0$ of $M\cap \Delta_g$ and we  denote by $M_{\Delta_g}$ the sublattice of $M$ generated by $M\cap (V-m_0)$, where $V-m_0$ is the translate of $V$ by $-m_0$. Then the translated polytope $\Delta_g-m_0$ is a full-dimensional lattice polytope in $M_{\Delta_g}\otimes_{\Z}\R$. To such a lattice polytope, one can attach a polarized projective toric $F$-variety $(\mathbb{P}_F(\Delta_g),\mathcal{L}(\Delta_g))$ with dense torus $\Spec F[M_{\Delta_g}]$. The toric variety $\mathbb{P}_F(\Delta_g)$ is defined by the inward normal fan of the polytope $\Delta_g-m_0$.
 We write $Z^o(g)$ for the effective Cartier divisor on $\Spec F[M_{\Delta_g}]$ defined by the equation $g=0$, and $Z(g)$ for the effective Cartier divisor on $\mathbb{P}_F(\Delta_g)$ defined by the global section $g/x^{m_0}$ of $\mathcal{L}(\Delta_g)$.
 These definitions are all independent of the choice of the point $m_0$.

\begin{rema}
The hypersurface defined by $g$ in the torus $\Spec F[M]$ is isomorphic to $Z^o(g)\times_F \mathbb{G}_{m,F}^{\ell}$, where $\ell$ denotes the codimension of $\Delta_g$ in $M_{\R}$.
 In particular, this hypersurface is stably birational to $Z^o(g)$.
\end{rema}

For every lattice polytope $\delta$ contained in $\Delta_g$, we set
$$g_{\delta}=\sum_{m\in M\cap \delta}d_m x^m.$$
If $\delta$ is a face of $\Delta_g$, then $g_{\delta}$ has Newton polytope $\delta$, and we can apply the above definitions to $g_{\delta}$ to obtain a toric variety $\mathbb{P}_F(\delta)$ of dimension $\dim(\delta)$ with subschemes $Z^o(g_{\delta})$ and $Z(g_{\delta})$. If $\delta$ is a vertex of $\Delta_g$ then these subschemes are empty; if $\delta$ has positive dimension, then they are non-empty and of pure codimension 1.

 \begin{defi}\label{defi:nondeg}
 We say that $g$ is {\em Newton non-degenerate} if, for every face $\delta$ of $\Delta_g$, the scheme $Z^o(g_{\delta})$ is smooth over $F$.
 \end{defi}

This definition is due to Kushnirenko (see \cite{kushnirenko} for related notions of Newton non-degeneracy). If $F$ has characteristic zero, then it follows from Bertini's theorem that a general Laurent polynomial $g$ with fixed Newton polytope $\Delta_g$ is Newton non-degenerate. The notion of Newton non-degeneracy was generalized to subvarieties of arbitrary codimension in algebraic tori by Tevelev in \cite{tevelev}; there these subvarieties are called {\em sch{\"o}n} (see also the proof of Proposition \ref{prop:nondeg}\eqref{it:toroidal} below).

\begin{prop}\label{prop:nondeg}
Let $g$ be a non-zero Laurent polynomial in $F[M]$.  Let $\pi\colon X\to \mathbb{P}_{F}(\Delta_g)$ be a proper birational toric morphism. Denote by $Y$ the inverse image of $Z(g)$ in $X$.
\begin{enumerate}
\item \label{it:dense} The scheme $Y$ does not contain any torus orbit in $X$, and $\pi^{-1}(Z^o(g))$ is schematically dense in $Y$. In particular, taking for $\pi$ the identity morphism, we see that $Z^o(g)$ is schematically dense in $Z(g)$.

\item \label{it:orbit}
The Laurent polynomial $g$ is Newton non-degenerate if and only if
the schematic intersection of $Y$ with each torus orbit in $X$ is smooth.

\item \label{it:toroidal}
Assume that $g$ is Newton non-degenerate.
Then we can cover $Y$ by open subschemes that admit an \'etale morphism to a toric $F$-variety. If $X$ is smooth over $F$, then $Y$ is smooth over $F$, as well.
\end{enumerate}
\end{prop}
\begin{proof}
\eqref{it:dense}
 If $\delta$ is a face of $\Delta_g$ and $\sigma$ is the corresponding cone in the normal fan of $\Delta_g$, then the intersection of $Z(g)$ with the torus orbit $O(\sigma)$ is isomorphic to $Z^o(g_{\delta})$; in particular, it is empty or of codimension 1 in $O(\sigma)$.
 Thus $Z(g)$ does not contain any torus orbit in $\mathbb{P}_{F}(\Delta_g)$. Since $\pi$ is a toric morphism and $Y$ is the inverse image of $Z(g)$ in $X$, the analogous property holds for $Y$ and $X$. It follows that $\pi^{-1}(Z^o(g))$ is topologically dense in $Y$.
 Since the toric variety $X$ is Cohen-Macaulay, the same holds for the Cartier divisor $Y$ on $X$. Thus by
  \cite[\href{https://stacks.math.columbia.edu/tag/083P}{Tag 083P}]{stacks-project}, $\pi^{-1}(Z^o(g))$ is also schematically dense in $Y$.

 \eqref{it:orbit} Let $\Sigma$ be the normal fan of $\Delta_g$ and let $\Sigma'$ be the complete  refinement of $\Sigma$
 associated with the proper birational toric morphism $X\to \mathbb{P}_{F}(\Delta_g)$.
 Let $\sigma'$ be a cone in $\Sigma'$ and denote by $\sigma$ the unique minimal cone of $\Sigma$ containing $\sigma'$. Let $\delta$ be the face of $\Delta_g$ dual to $\sigma$.
 Then the intersection of $Z(g)$ with the orbit $O(\sigma)$ is isomorphic to $Z^o(g_{\delta})$,
 and the intersection of $Y$ with $O(\sigma')$ is the inverse image of $Z^o(g)$ under the torus fibration $O(\sigma')\to O(\sigma)$. Thus $Z^o(g_\delta)$ is smooth if and only if
 $Y\cap O(\sigma')$ is smooth.

\eqref{it:toroidal} This result is well-known to experts; it follows from the fact that $Z^o(g)$ is a
sch\"on hypersurface in the torus $\Spec F[M_{\Delta_g}]$, in the sense of Tevelev -- see in particular Theorem 1.4 and the subsequent remark in \cite{tevelev}. If $X$ is smooth, then it was proved already in \cite{Khovanskii} that $Y$ is smooth, as well. For the reader's convenience, we will give a direct argument that covers both the smooth and the singular case.
 In the language of logarithmic geometry, we will show that $Y$ is log smooth when we endow it with the pullback of the canonical log structure on the toric variety $X$. 
 
 Replacing $M$ by $M_{\Delta_g}$ if necessary, we may assume that $\Delta_g$ has maximal dimension in $M_{\R}$, so that $M=M_{\Delta_g}$.
 Let $y$ be a point of $Y$. Let $\Sigma$ be the normal fan of $\Delta_g$ and let $\Sigma'$ be the complete refinement of $\Sigma$
 associated with the proper birational toric morphism $X\to \mathbb{P}_{F}(\Delta_g)$.
    We denote  by $\sigma'$  the unique cone in $\Sigma'$ such that $y$ is contained in the torus orbit $O(\sigma')$.
        We set $U=\Spec F[(\sigma')^{\vee}\cap M]$; this is an affine toric chart of $X$  with $O(\sigma')$ as its minimal orbit.
   We will write $U$ as a product of $O(\sigma')$ with an affine toric variety $V$ and then prove that the projection $U\to V$ restricts to a morphism $Y\cap U\to V$ that is smooth at $y$.
   
   We choose a splitting
  $$(\sigma')^{\vee}\cap M\cong P\oplus M'$$ where $P$ is a finitely generated saturated integral monoid such that $P^{\times}=\{0\}$, and $M'$ is the lattice of invertible elements in
  the monoid $(\sigma')^{\vee}\cap M$.
 This choice induces an isomorphism
$$U\cong \Spec F[P]\times_F \Spec F[M'].$$
 Then $V=\Spec F[P]$ is an affine toric variety, and the fiber of the projection morphism $U\to V$ over the zero-dimensional orbit $o$ of $V$ is precisely the orbit $O(\sigma')$.

 It follows from
 \cite[\href{https://stacks.math.columbia.edu/tag/00MG}{Tag 00MG}]{stacks-project} that
 the morphism $Y\cap U\to Z$ is flat at $y$. The fiber of this morphism over $o$ coincides with $Y\cap O(\sigma')$. In particular, this fiber is smooth, by \eqref{it:orbit}.
 Thus the morphism $Y\cap U\to V$ is smooth at $y$, and, locally around $y$, there exists an \'etale morphism from $Y$ to the toric variety $V\times_F \A^\ell_F$, for some $\ell\geq 0$. If $X$ is smooth then the monoid $P$ is free, so that $V$ is an affine space and $Y$ is smooth at $y$.
\end{proof}

\begin{prop}\label{prop:sbir}
Assume that $F$ has characteristic zero. Let $g$ be a Laurent polynomial in $F[M]$ that is Newton non-degenerate. Then
$$\sbmap (Z(g_{\delta}))=\sbir{Z(g_{\delta})}=\sbir{Z^o(g_{\delta})}$$ for every face $\delta$ of $\Delta_g$.
\end{prop}
\begin{proof}
 The equality
 $\sbir{Z(g_{\delta})}=\sbir{Z^o(g_{\delta})}$ follows immediately from the fact that $Z^o(g_{\delta})$ is dense in
 $Z(g_{\delta})$, by Proposition \ref{prop:nondeg}\eqref{it:dense}.
 Thus it is enough to show that $\sbmap (Z(g_{\delta}))=\sbir{Z(g_{\delta})}$.
It suffices to prove this result for $\delta=\Delta_g$, since it is obvious from the definition that each of the polynomials $g_{\delta}$ is Newton non-degenerate if this holds for $g$.

 By Proposition \ref{prop:nondeg}, the scheme $Z(g)$ has {\em strictly toroidal singularities} in the sense of Example 4.2.6(3) in \cite{NiSh}.
 This implies that we can find a resolution of singularities
 $Y\to Z(g)$ such that $[Y]\equiv [Z(g)]\mod \LL$ in $\Gro(\Var_F)$ (in fact, this then holds for {\em every} resolution of singularities; see the discussion following Definition 4.2.4 in \cite{NiSh}). Since $Y$ is smooth and proper, we conclude by Theorem \ref{theo:lalu} that
 \[\sbir{Z(g)}=\sbir{Y}=\sbmap(Y) = \sbmap(Z(g)).\qedhere\]
\end{proof}

\subsection{Stably irrational polytopes}
Let $\Delta$ be a lattice polytope in $M_{\R}$. For every field $F$, the Laurent polynomials in $F[M]$ with Newton polytope $\Delta$ are parameterized by a dense Zariski open subset $U$ of the $F$-vector space of maps $\Delta\cap M\to F$. This open subset is defined by the condition that the images of the vertices of $\Delta$ are different from zero. We say that a property holds for every very general Laurent polynomial in $F[M]$ with Newton polytope $\Delta$ if it holds for all Laurent polynomials parameterized by a countable intersection of non-empty Zariski open subsets of $U$ (note that this is an empty statement when the field $F$ is countable).

\begin{defi}\label{defi:stablyratpol}
We say that $\Delta$ is {\em stably irrational} if $\Delta$ has dimension at least $2$ and, for every algebraically closed field $F$ of characteristic zero, and for every very general polynomial $g$ in $F[M]$ with Newton polytope $\Delta$, the hypersurface $Z^o(g)$ in $\Spec F[M]$ is not stably rational. Otherwise, we say that $\Delta$ is {\em stably rational}.
\end{defi}

 In particular, zero-dimensional and one-dimensional polytopes are stably rational by definition. In these cases,
 $Z^o(g)$ is empty or a finite set of points, respectively. If $\Delta$ has dimension at least $2$,  and $g$ is as in the definition, then $Z^o(g)$ is integral by Bertini's theorem. Thus, in all dimensions, we conclude that $\Delta$ is stably rational if and only if for every $g$ as in the definition, all connected components of $Z^o(g)$ are stably rational.
  Note that the stable rationality of $\Delta$ does not depend on the embedding of $\Delta$ in $M_{\R}$:
the isomorphism class of the hypersurface $Z^o(g)$ is invariant under unimodular affine transformations of $M$.

\begin{exam}\label{exam:hypersurf}
For every positive integer $d$, we denote by $d\Delta_{n+1}$ the dilatation with factor $d$ of the unimodular $(n+1)$-dimensional simplex:
$$d\Delta_{n+1}=\left\{u\in \R^{n+1}_{\geq 0}\,\,\big|\,\,\sum_{i=1}^{n+1} u_i\leq d\right\}.$$
 Then for every infinite field $F$, the hypersurfaces $Z(g)$ in $\mathbb{P}_F(d\Delta_{n+1})=\mathbb{P}^{n+1}_F$ defined by polynomials $g$ with Newton polytope $d\Delta_{n+1}$ form a dense open subset of the space of degree $d$ hypersurfaces in $\mathbb{P}^{n+1}_F$. Thus $d\Delta_{n+1}$ is stably irrational if and only if $n\geq 1$ and, over every algebraically closed field $F$ of characteristic zero, a very general degree $d$ hypersurface in $\mathbb{P}^{n+1}_F$ is not stably rational.
\end{exam}

\begin{exam}\label{exam:bideg2}
The lattice polytope $\Delta=2\Delta_2\times 2\Delta_3$ in $\R^{5}$ is stably irrational.
 Indeed, for every infinite field $F$, the hypersurfaces $Z(g)$ in $\mathbb{P}_F(\Delta)=\mathbb{P}^{2}_F\times_F \mathbb{P}^{3}_F$ defined by polynomials $g$ with Newton polytope $\Delta$ form a dense open subset of the space of bidegree $(2,2)$ hypersurfaces in $\mathbb{P}^{2}_F\times_F \mathbb{P}^{3}_F$ defined over $F$. When $F$ has characteristic $0$, a very general member of this family is stably irrational by \cite{HPTbideg2} (see in particular the summary of results in \S7).
\end{exam}

\begin{exam}\label{exam:prim}
Assume that $\Delta$ has lattice width $1$ in the lattice $M_{\Delta}$. Then $\Delta$ is stably rational.
  Indeed, replacing $M$ by $M_{\Delta}$, we may assume that $\Delta$ has full dimension $n+1$. Our assumption on the lattice width means that we can find a primitive vector $\ell$ in the dual lattice $M^{\vee}$ such that $\ell(\Delta\cap M)=\{a,a+1\}$ for some integer $a$.
Translating $\Delta$ in $M_{\R}$ if needed, we may assume that $a=0$.
 Choosing a basis of $M^{\vee}$ that contains $\ell$, we see that for every field $F$
and every polynomial $g$ in $F[M]$ with Newton polytope $\Delta$, the hypersurface $Z^o(g)$ is isomorphic to a hypersurface in $\mathbb{G}_{m,F}^{n+1}$ defined by an equation that is linear in one of the variables. It follows that $Z^o(g)$ is rational.
\end{exam}

\begin{exam}\label{hptpolytope}
Consider the Hassett--Pirutka--Tschinkel quartic 
\begin{equation}\label{hptquartic}
g=x_1x_2x_3^2+x_1 x_4^2+x_2 x_5^2+x_1^2+x_2^2-2x_1x_2-2x_1-2x_2+1
\end{equation}The Newton polytope $\Delta$ of $g$ is the convex hull of the six points in $\RR^5$
\begin{center}
\begin{tabular}{ccc}
(0, 0, 0, 0, 0), & (2, 0, 0, 0, 0), & (0, 2, 0, 0, 0),\\ (1, 1, 2, 0, 0), & (1, 0, 0, 2, 0), & (0, 1, 0,
       0, 2).
\end{tabular}\end{center}
It follows from \cite{HPTbideg2} and \cite{schreieder} that this hypersurface is stably irrational over any algebraically closed field $F$ of characteristic $0$. Applying Proposition 3.1 in \cite{schreieder}, it also follows that a very general hypersurface over $F$ with Newton polytope $\Delta$ does not admit a decomposition of the diagonal. In particular, the polytope $\Delta$ is stably irrational.
\end{exam}
 In order to confirm that a given lattice polytope is stably irrational using Definition \ref{defi:stablyratpol}, we {\em a priori} need to test infinitely many hypersurfaces over infinitely many base fields. The following proposition guarantees that it suffices to test one sufficiently general hypersurface over one algebraically closed base field.

\begin{prop}\label{prop:stabyratpol}
Let $F_0$ be a field of characteristic $0$, and let
 $W$ be a reduced $F_0$-scheme of finite type. Suppose that there exist an algebraically closed extension $F$ of $F_0$ and a  Newton non-degenerate polynomial $g$ in $F[M]$ with Newton polytope $\Delta$ such that $Z^o(g)$ is not stably birational to $W\times_{F_0} F$. Then for {\em every} algebraically closed extension $F'$ of $F_0$ and every very general polynomial $h$ in $F'[M]$ with Newton polytope $\Delta$, the hypersurface $Z^o(h)$ is not stably birational to $W\times_{F_0}F'$.

In particular, if $\dim(\Delta)\geq 2$, then the polytope $\Delta$ is stably irrational if and only if there exist an algebraically closed field $F$ of characteristic zero and a polynomial $g$ in $F[M]$ with Newton polytope $\Delta$ such that $g$ is Newton non-degenerate and $Z^o(g)$ is not stably rational.

If $\mathbb{P}_{\Q}(\Delta)$ is smooth, then the same results hold if we only suppose that $Z(g)$ is smooth, instead of assuming that $g$ is Newton non-degenerate.
\end{prop}
\begin{proof}
 Replacing $M$ by the sublattice $M_{\Delta}$, we may assume that $\Delta$ has full dimension $n+1$.
Assume that we can find $F$ and $g$ as in the statement.
 Let $\pi\colon X\to \mathbb{P}_{F_0}(\Delta)$ be a toric resolution of the projective toric variety $\mathbb{P}_{F_0}(\Delta)$ over $F_0$. If $\mathbb{P}_{F_0}(\Delta)$ is already smooth, we take for $\pi$ the identity morphism.

 Let $\mathcal{L}(\Delta)$ be the canonical polarization on $\mathbb{P}_{F_0}(\Delta)$. Then $\mathcal{L}(\Delta)$ is generated by global sections, so that the same holds for $\pi^*\mathcal{L}(\Delta)$. Since $\mathbb{P}_{F_0}(\Delta)$ is normal, the natural map
  $$H^0(\mathbb{P}_{F_0}(\Delta),\mathcal{L}(\Delta))\to H^0(X,\pi^*\mathcal{L}(\Delta))$$ is an isomorphism.

 Let $F'$ be an algebraically closed extension of $F_0$, and let  $h$ be a Laurent polynomial in $F'[M]$ with Newton polytope $\Delta$. We set $Z'(h)=Z(h)\times_{\mathbb{P}_{F_0}(\Delta)}X$.
  Then $Z^o(h)$ is schematically dense in $Z'(h)$, by Proposition
  \ref{prop:nondeg}\eqref{it:dense}.
 Denote by $\mathcal{L}'$ the pullback of $\pi^*\mathcal{L}(\Delta)$ to $X\times_{F_0}F'$.
  Then $Z'(h)$ is the member of the linear system $|\mathcal{L}'|$ associated with $$h\in H^0(X\times_{F_0}F',\mathcal{L}')\cong H^0(\mathbb{P}_{F_0}(\Delta),\mathcal{L}(\Delta))\otimes_{F_0}F'.$$

 Consider the universal family
$$\psi\colon \cY \hookrightarrow X\times_{F_0}  \mathbb{P}H^0(\mathbb{P}_{F_0}(\Delta),\mathcal{L}(\Delta))$$
of the linear system $|\pi^*\mathcal{L}(\Delta)|$ on $X$.  There is a dense open subset $U$ of $\mathbb{P}H^0(\mathbb{P}_{F_0}(\Delta),\mathcal{L}(\Delta))$ such that the morphism $$\psi^{-1}(\Spec F_0[M]\times_{F_0}U)\to U$$ is the universal family of hypersurfaces in $\Spec F_0[M]$ with Newton polytope $\Delta$.
 By Bertini's theorem, the smooth fibers of the family $$\theta\colon \cY\times_{H^0(\mathbb{P}_{F_0}(\Delta),\mathcal{L}(\Delta))}U\to U$$ are parameterized by a dense open subset $V$ of $U$. By Corollary \ref{coro:verygen}, the locus of geometric fibers of the family
 $\theta^{-1}(V)\to V$ that are not stably birational to $W$ is a countable intersection of open subsets of $V$ (take a smooth and proper birational model $W'$ of $W$ and apply Corollary \ref{coro:verygen} to $\theta^{-1}(V)\to V$  and $W'\times_{F_0} V\to V$ ). This locus is non-empty: by Proposition \ref{prop:nondeg},
 we know that $Z'(g)$ is smooth, so that it corresponds to a geometric fiber of
  $\theta^{-1}(V)\to V$ that is not stably birational to $W$.
 It follows that a very general geometric fiber of the universal family
$$\psi^{-1}(\Spec F_0[M]\times_{F_0}U)\to U$$
 is not stably birational to $W$.
\end{proof}

\begin{exam}\label{exam:doublequar}
Let $\Delta$ be the lattice polytope in $\R^{n+1}$ defined by
$$\Delta=\{u\in \R^{n+1}_{\geq 0}\,|\,u_1+\ldots+u_n+2u_{n+1}\leq 4\},$$ and let $F$ be an algebraically closed field of characteristic $0$. Then a general quartic double  $n$-fold in $\mathbb{P}_F(\Delta)=\mathbb{P}_F(1,\ldots,1,2)$
 has Newton polytope $\Delta$, and it is Newton non-degenerate.
 A very general quartic double $n$-fold is stably irrational
 for $n=3$ \cite{voisin} and $n=4$ \cite{HPTdouble}. Thus $\Delta$ is stably irrational for $n\in \{3,4\}$. The $n=3$ case also follows from the classical nodal Artin-Mumford example \cite{Artin-Mumford} and Theorem \ref{theo:vol}; see Example 4.3.2 in \cite{NiSh}.
\end{exam}

\subsection{Newton subdivisions}\label{ss:subdiv}
 We now consider polynomials over the Puiseux series field $K$.
Let $f=\sum_{m\in M}c_m x^m\in K[M]$ be a non-zero Laurent polynomial with coefficients $c_m$ in $K$.
 Let $$\varphi_f \colon \Delta_f\to \R$$ be the lower convex envelope of the function
 $$\Supp(f)\to \Q,\,m\mapsto \ord_t(c_m).$$
  Thus $\varphi_f$ is the largest convex function such that
  $\varphi_f(m)\leq \ord_t(c_m)$ for every $m$ in $\Supp(f)$. Then $\varphi_f$ is piecewise linear,
  and the maximal domains where it is affine define a polyhedral subdivision $\mathscr{P}_f$ of $\Delta_f$, which is called the Newton subdivision of $\Delta_f$. It is easy to see that all the vertices of $\mathscr{P}_f$ lie in $\Supp(f)$, and $\varphi_f(m)=\ord_t(c_m)$ for every such vertex $m$.

 For every face $\delta$ of $\mathscr{P}_f$, we  set
 $$\overline{f}_{\delta}= \sum_{m\in \delta\cap \Supp(f)}\overline{t^{-\varphi_f(m)}c_m}x^{m}\in k[M_\delta]$$
 where $\overline{a}$ denotes the image of an element $a\in R$ under the reduction morphism $R\to k$ that sends $t^q$ to $0$ for all positive rational numbers $q$. Thus $\overline{f}_{\delta}$ is obtained from $f$ by selecting the terms corresponding to lattice points $m\in M\cap \delta$ such that $\varphi_f(m)=\ord_t(c_m)$, and replacing each coefficient $c_m$ by its leading coefficient.
      We say that $f$ is {\em sch{\"o}n} if $Z^o(\overline{f}_{\delta})$ is smooth over $k$, for all faces $\delta$ of $\mathscr{P}_f$.

\begin{prop}\label{prop:schon-nondeg}
   If $f$ is sch{\"o}n, then it is Newton non-degenerate.
\end{prop}
\begin{proof}
 The assertion can easily be deduced from the general theory of tropical compactifications of subvarieties of tori \cite{tevelev, luxton-qu}.  Our definition of ``sch{\"o}n'' can be rephrased as the property that all the initial degenerations of $Z^o(f)$ are smooth over $k$. By Proposition 3.9 in \cite{helm-katz}, this is equivalent to the existence of a $\Spec R[M]$-toric $R$-scheme  $\cX$ such that the schematic closure $\cY$ of $Z^o(f)$ in $\cX$ is proper over $R$ and the multiplication morphism
 $$\Spec R[M]\times_R \cY \to \cX$$ is smooth and surjective. Passing to the generic fiber, we obtain a $\Spec K[M]$-toric $K$-scheme $X=\cX_K$ such that the schematic closure $Y$ of $Z^o(g)$ in $X$ is proper over $K$ and the multiplication morphism
 $$\Spec K[M]\times_K Y \to X$$ is smooth and surjective.
  Thus $Z^o(f)$ is sch{\"o}n in the sense of \cite{tevelev}, which is equivalent to the property that $f$ is Newton non-degenerate -- see the remark following Theorem 1.4 in \cite{tevelev}.
 The equivalence follows from the fact that
    the fibers
 of $$\Spec K[M]\times_K Y \to X$$ over $K$-rational points of $X$ are torus bundles over
 the schemes $Z^o(f_{\tau})$, where $\tau$ runs through the faces of $\Delta_f$.
\end{proof}

\begin{rema}
If $c_m\in R^{\times}$ for all $m$ in $\Supp(f)$, then $\varphi_f=0$ and $\mathscr{P}_f$ is the trivial subdivision of $\Delta_f$. In this case, $f$ is sch{\"o}n if and only if it is Newton non-degenerate.
\end{rema}

\subsection{Viro's patchworking polynomial}\label{ss:viro}
 For the applications we have in mind, it is often convenient to construct a polynomial $f$ from given tropical data by means of some reverse engineering.
 Such an approach was used already around 1980 in Viro's patchworking technique to construct real hypersurfaces with prescribed topological properties \cite{viro, viro-thesis}.
 
 Let $\Delta$ be a lattice polytope in $M_{\R}$. A polyhedral subdivision of $\mathscr{P}$ is called {\em integral} if all its faces are lattice polytopes. It is called {\em regular} if its maximal faces are the domains of linearity of some convex piecewise linear function $\varphi\colon \Delta \to \R$. In that case, we will say that $\varphi$ {\em induces} the subdivision $\mathscr{P}$. Assume that $\mathscr{P}$ is regular and integral. Then we can choose $\varphi$ such that it takes rational values at the lattice points in $\Delta$.  For every lattice point $m$ in $M\cap \Delta$, we choose an element $d_m$ in $k$ such that $d_m\neq 0$ whenever $m$ is a vertex of $\mathscr{P}$. Then, for every face $\delta$ of $\mathscr{P}$, the polynomial $g_\delta=\sum_{m\in M\cap \delta}d_m x^m$ in $k[M]$ has Newton polytope $\delta$.
 If we set $$f=\sum_{m\in M\cap \Delta}t^{\varphi(m)}d_mx^m\in K[M]$$
 then $f$ has Newton polytope $\Delta_f=\Delta$ and Newton subdivision $\mathscr{P}_f=\mathscr{P}$, and $\overline{f}_\delta=g_{\delta}$ for every face $\delta$ of $\mathscr{P}$.
   If we choose the coefficients $d_m$ in such a way that $Z^o(g_{\delta})$ is smooth for all $\delta$, then $f$ is sch{\"o}n. This is always the case for a general choice of coefficients $d_m$, by Bertini's theorem.

  \subsection{The stable birational volume}
 The following theorem provides an efficient way to compute the stable birational volume of a general hypersurface with fixed Newton polytope.

\begin{theo}\label{theo:nondeg}
Let $f\in K[M]$ be a Laurent polynomial with coefficients in $K$.
 Assume that $f$ is sch{\"o}n.
Then the stable birational volume of $Z^o(f)$ satisfies
$$\Volsb(\sbir{Z^o(f)})=(-1)^{\dim(\Delta_f)}\sum_{\delta \nsubseteq  \partial \Delta_f} (-1)^{\dim(\delta)}\sbir{Z^o(\overline{f}_{\delta})}$$
in $\Z[\SB_k]$, where $\delta$ runs through the faces of the Newton subdivision $\mathscr{P}_f$ that are not contained in the boundary of $\Delta_f$.
\end{theo}
\begin{proof}
 Replacing $M$ by the lattice $M_{\Delta_f}$, we can reduce to the case where $\Delta_f$ has full dimension $n+1$.
 We will deduce the result from the tropical formula for the motivic volume of $Z(f)$ in
Proposition 3.13 of \cite{NPS}. The formula for the motivic volume can be specialized to a formula for the stable birational volume by applying Larsen and Lunts' morphism $\sbmap$ from Theorem \ref{theo:lalu}, because of the commutative diagram at the end of Section \ref{sec:sbir}.

The formula for the motivic volume in Proposition 3.13 of \cite{NPS} is stated in terms of the {\em tropicalization} of $Z^o(f)$, rather than the Newton polytope of $f$. The connection between these is explained in Proposition 3.1.6 of \cite{MS}: the tropicalization of $Z^o(f)$ is the codimension $1$ skeleton of the polyhedral subdivision $\mathscr{Q}$ of $N_{\R}=\Hom(M,\R)$ dual to $\varphi_f$.
 More precisely, $\mathscr{Q}$ is obtained by taking the fan in $N_{\R}\times \R$ normal to the polytope 
 $$\widetilde{\Delta}=\{(u,v)\in \Delta\times \R\,|\,v\geq \varphi_f(u)\},$$
 slicing it with the hyperplane defined by $v=1$, and projecting to $N_{\R}$.
 For every element $d$ in $\{0,\ldots,n+1\}$, there is a canonical inclusion-reversing bijection $\delta\mapsto \delta^{\vee}$ between the $d$-dimensional faces $\delta$ of $\mathscr{P}_f$ and the $(n+1-d)$-dimensional faces $\delta^{\vee}$ of $\mathscr{Q}$. In particular, the faces of the tropicalization of $Z^o(f)$ correspond to the faces of $\mathscr{P}_f$ of positive dimension.
 In the notation of Proposition 3.13 in  \cite{NPS}, if we set $X=Z^o(f)$ then the $k$-scheme $\mathcal{X}_{k}(\gamma)$ is precisely our $k$-scheme $Z^o(\overline{f}_{\delta})$ when $\gamma=\delta^{\vee}$. Thus the formula in that proposition implies that
\begin{equation}\label{eq:motvol}
\Vol(Z(f))\equiv \sum_{\delta}[Z^o(\overline{f}_{\delta})] \mod \LL
\end{equation} in $\Gro(\Var_k)$, where $\delta$ runs through the faces of $\mathscr{P}_f$ (we do not need to exclude the zero-dimensional faces $\delta$ because, for those faces, we have $Z^o(\overline{f}_{\delta})=\emptyset$).

 We will now use an inclusion-exclusion argument to show that
 \begin{equation}\label{eq:inclexcl}
 \sum_{\delta}[Z^o(\overline{f}_{\delta})]=\sum_{\varepsilon\nsubseteq  \partial \Delta_f}(-1)^{n+1-\dim(\varepsilon)}[Z(\overline{f}_{\varepsilon})]
 \end{equation}
 in $\Gro(\Var_k)$, where $\varepsilon$ runs through the faces of $\mathscr{P}_f$ that are not contained in the boundary of $\Delta_f$.  For every face $\delta$ of $\mathscr{P}_f$, we can write $Z(\overline{f}_{\delta})$ as the disjoint union of the $k$-schemes $Z^o(\overline{f}_{\delta'})$ where $\delta'$ runs through the faces of $\delta$. It follows that
$$\sum_{\varepsilon\nsubseteq  \partial \Delta_f}(-1)^{n+1-\dim(\varepsilon)}[Z(\overline{f}_{\varepsilon})]=\sum_{\varepsilon\nsubseteq  \partial \Delta_f}(-1)^{n+1-\dim(\varepsilon)}\left( \sum_{\varepsilon'\leq \varepsilon}[Z^o(\overline{f}_{\varepsilon'})]\right)$$ in $\Gro(\Var_k)$, where $\varepsilon'$ runs through the faces of $\varepsilon$. Reordering terms, we see that it suffices to show that, for every face $\delta$ of $\mathscr{P}_f$, we have
$$ \sum_{\varepsilon\in A_{\delta}}(-1)^{n+1-\dim(\varepsilon)}=1 $$
where $A_{\delta}$ is the set of faces of $\mathscr{P}_f$ that contain $\delta$ and are not contained in the boundary of $\Delta_f$.
 In terms of the dual polyhedral decomposition $\mathscr{Q}$, this is equivalent to showing that for every face $\delta^{\vee}$ of $\mathscr{Q}$, we have
$$\sum_{\varepsilon^{\vee}\in B_{\delta^{\vee}}}(-1)^{\dim(\varepsilon^{\vee})}=1$$ where $B_{\delta^{\vee}}$ is the set of bounded faces of $\delta^{\vee}$. But this expression is nothing but the Euler characteristic of the union of the bounded faces of $\delta^{\vee}$, which is equal to $1$ because this union is a contractible compact topological space.

The result now follows from equalities \eqref{eq:motvol} and \eqref{eq:inclexcl}, together with the fact that the isomorphism
$$\sbmap \colon \Gro(\Var_k)/(\LL)\to \Z[\SB_k]$$ maps $[Z(f)]$ to $\sbir{Z^o(f)}$
and $[Z(\overline{f}_{\delta})]$ to $\sbir{Z^o(\overline{f}_{\delta})}$, for every face $\delta$ of $\mathscr{P}_f$, by Proposition \ref{prop:sbir}.
\end{proof}

\begin{theo}\label{theo:main}
 Let $n$ be a positive integer.
Let $\Delta$ be a lattice polytope in $\R^{n+1}$ of dimension at least $2$, and let $\mathscr{P}$ be a regular integral  polyhedral subdivision of $\Delta$. For every $m$ in $\Delta\cap \Z^{n+1}$, we choose an element $d_m$ in $k$ such that $d_m\neq 0$ whenever $m$ is a vertex of $\mathscr{P}$. For every face $\delta$ in $\mathscr{P}$, we set
$$g_{\delta}=\sum_{m\in \Z^{n+1}\cap \delta}d_mx^m\in k[\Z^{n+1}].$$

Suppose that, for every face $\delta$ in $\mathscr{P}$, the scheme $Z^o(g_{\delta})$ is smooth over $k$. Assume moreover that
\begin{equation}\label{eq:assum}
\sum_{\delta \nsubseteq \partial \Delta} (-1)^{\dim(\delta)}\sbir{Z(g_{\delta})}\neq (-1)^{n+1}\sbir{\Spec k}
\end{equation}
in $\Z[\SB_k]$, where the sum is taken over all faces $\delta$ of $\mathscr{P}$ that are not contained in the boundary of $\Delta$.
Then the polytope $\Delta$ is stably irrational.
\end{theo}
\begin{proof}
By Proposition \ref{prop:stabyratpol}, it is enough to
     find a Laurent polynomial $f$ in $K[\Z^{n+1}]$ with Newton polytope $\Delta$ such that $f$ is Newton non-degenerate and $Z^o(f)$ is stably irrational.
We choose a convex piecewise linear function $\varphi\colon \Delta \to \R$ that induces the regular subdivision $\mathscr{P}$.
 We set $$f=\sum_{m\in \Z^{n+1}\cap \Delta}t^{\varphi(m)}d_mx^m\in K[\Z^{n+1}].$$
 Then $f$ has Newton polytope $\Delta_f=\Delta$ and {Newton subdivision} $\mathscr{P}_f=\mathscr{P}$, and $\overline{f}_\delta=g_{\delta}$ for every face $\delta$ of $\mathscr{P}$.
   By our assumption that $Z^o(g_{\delta})$ is smooth for all faces $\delta$ of $\mathscr{P}$, the Laurent polynomial $f$ is sch{\"o}n; in particular, it is Newton non-degenerate, by Proposition \ref{prop:schon-nondeg}.
 Now it follows from Theorem \ref{theo:nondeg} and our assumption \eqref{eq:assum} that
 $$\Volsb(\sbir{Z^o(f)})\neq \sbir{\Spec k}$$ in $\Z[\SB_k]$. Thus
 $\sbir{Z^o(f)}\neq \sbir{\Spec K}$ in $\Z[\SB_K]$, which means that $Z^o(f)$ is not stably rational.
\end{proof}

\begin{coro}\label{coro:irratpol}
 Let $n$ be a positive integer.
Let $\Delta$ be a lattice polytope in $\R^{n+1}$, and let $\mathscr{P}$ be a regular integral  polyhedral subdivision of $\Delta$. Assume that $\mathscr{P}$ has a stably irrational face $\delta_0$ that is not contained in the boundary of $\Delta$ and such that all the other faces of $\mathscr{P}$ not contained in the boundary of $\Delta$ are stably rational. Then $\Delta$ is stably irrational.
\end{coro}
\begin{proof}
We choose a very general tuple $(d_m)_m$ of elements in $k$ indexed by the lattice points $m$ in $\Delta$, and we define the polynomials $g_\delta$ as in the statement of Theorem \ref{theo:main}. Then 
$Z^o(g_\delta)$ is smooth over $k$ for every face $\delta$ of $\mathscr{P}$. In the left hand side of equation \eqref{eq:assum}, the term corresponding to $\delta_0$ is the only term in the sum that is not a multiple of $\sbir{\Spec k}$, because all the other faces that meet the relative interior of $\Delta$ are stably rational. Thus the inequation \eqref{eq:assum} is satisfied, and $\Delta$ is stably irrational.
\end{proof}
In order to prove that a given polytope is stably irrational, one is therefore led to search for suitable subdivisions of it containing a smaller polytope which is already known to be stably irrational. Note that the subdivision is required to be regular; in general, checking whether a given subdivision is regular can be difficult. For most of the applications in this paper, we will specify the subdivision by writing down the piecewise linear function $\varphi\colon \Delta \to \R$ directly.

\subsection{Tropical degenerations}\label{ss:tropdeg}
This section is not strictly necessary for the applications in this article, but it provides a geometric explanation for the results we obtained in Theorems \ref{theo:nondeg} and \ref{theo:main}.
We will show how the data of Theorem \ref{theo:main} give rise to an explicit degeneration of a Newton non-degenerate hypersurface with Newton polytope $\Delta$ into pieces that are described by the hypersurfaces $Z^o(g_{\delta})$.

 Let $f=\sum_{m\in M}c_m x^m\in K[M]$ be a Laurent polynomial with coefficients $c_m$ in $K$. Denote by $\Delta_f$ the Newton polytope of $f$, and by $\mathscr{P}_f$ its Newton subdivision.
 Let $$\varphi_f\colon \Delta_f\to \R$$ be the lower convex envelope of the function
 $$\Supp(f)\to \Q,\,m\mapsto \ord_t(c_m).$$
Assume that $f$ is sch{\"o}n.
 Then we can construct a degeneration of $Z(f)$ by considering the {\em Mumford degeneration} of the projective variety $\mathbb{P}_K(\Delta_f)$ induced by the data $(\mathscr{P}_f,\varphi_f)$ (see for instance Example 3.6 in \cite{gross-tropical}), and by taking the schematic closure of $Z(f)$ inside the Mumford degeneration.

Let $N=\Hom(M,\Z)$ be the dual lattice of $M$. 
Let $\widetilde{\Delta}$ be the polyhedron in $M_{\R}\times \R$ defined by
$$\widetilde{\Delta}=\{(u,v)\in \Delta\times \R\,|\,v\geq \varphi_f(u)\}.$$
We choose a positive integer $e$ such that $\varphi_f(u)$ lies in $(1/e)\Z$ for every vertex $u$ of $\mathscr{P}_f$, and such that the slopes of $\varphi_f$ lie in $(1/e)N$.
 Then $\widetilde{\Delta}$ is a lattice polyhedron with respect to the lattice $\Z^{n+1}\times (1/e)\Z$, and its asymptotic cone is given by $\{0\}\times \R_{\geq 0}$. Thus $\widetilde{\Delta}$ defines a toric variety $\mathbb{P}_{k}(\widetilde{\Delta})$ with a projective morphism to $\A^1_k=\Spec k[t^{1/e}]$. We set $\cX=\mathbb{P}_{k}(\widetilde{\Delta})\times_{k[t^{1/e}]}R$; this definition does not depend on the choice of $e$. The generic fiber $\cX_K$ is precisely $\mathbb{P}_{K}(\Delta_f)$, and the special fiber $\cX_k$ is a union of the toric varieties $\mathbb{P}_k(\delta)$ where $\delta$ runs through the faces of maximal dimension in $\mathscr{P}_f$. These toric varieties intersect along toric strata according to the combinatorial structure of the subdivision $\mathscr{P}_f$: the closed strata in $\cX_k$ correspond canonically to the faces $\delta$ in $\mathscr{P}_f$, and this correspondence is inclusion-preserving.


 Since $Z(f)$ is a closed subscheme of $\cX_K=\mathbb{P}_K(\Delta_f)$, we can consider its schematic closure in $\cX$, which we denote by $\cY$. The schematic intersection of $\cY$ with the open stratum in $\cX_k$ corresponding to a face $\delta$ in $\mathscr{P}_f$ is precisely $Z^o(f_{\delta})$. Thus if $f$ is sch{\"o}n, all these intersections are smooth.
  It follows that, if we endow $\cY$ with the log structure induced by the toric boundary of $\cX$ (that is, the union of the special fiber $\cX_k$ with the closure of the toric boundary of $\cX_K=\mathbb{P}_K(\Delta_f)$) and we endow $\Spec R$ with its standard log structure (induced by the closed point), then the morphism $\cY\to \Spec R$ is log smooth. One can now deduce Theorem \ref{theo:nondeg} from the formula for the motivic volume in terms of log smooth models in Theorem A.3.9 of \cite{NiSh}. We omit the details of the argument, as we will not use these facts in this paper.

Let us illustrate this geometric picture by means of an example that will appear again in Corollary \ref{coro:hyperplane} and Theorem \ref{theo:slope}.

\begin{exam}\label{exam:linear}
Let $n$ and $d$ be positive integers and assume that $d\geq 2$. Let $f_0$ be a general homogeneous polynomial of degree $d$ in $k[z_1,\ldots,z_{n+1}]$, and let $f_1$ be a general homogeneous polynomial of degree $d-1$ in $k[z_0,\ldots,z_{n+1}]$. We set $f=tf_0+z_0f_1$;
 this is a homogeneous polynomial of degree $d$ in $K[z_0,\ldots,z_{n+1}]$.
   The Newton polytope $\Delta_f$ is the dilatation with factor $d$ of the unimodular $(n+1)$-dimensional simplex:
 $$\Delta_f=\{(u_0,\ldots,u_{n+1})\in \R^{n+2}_{\geq 0}\,|\,u_0+\ldots+u_{n+1}=d\}.$$
 The function $\varphi_f$ is given by
$\varphi_f=\max \{0,1-u_0\}$.
 The Newton subdivision $\mathscr{P}_f$ of $\Delta_f$ has two maximal cells, namely,
 $$\begin{array}{lll}
\delta_{\leq}&=& \{(u_0,\ldots,u_{n+1})\in \Delta_f\,|\,u_0\leq 1\},
\\[1.5ex] \delta_{\geq}&=& \{(u_0,\ldots,u_{n+1})\in \Delta_f\,|\,u_0\geq 1\}.
  \end{array}$$
  They intersect along the codimension 1 face
  $$\begin{array}{lll}\delta_{=}&=&\{(u_0,\ldots,u_{n+1})\in \Delta_f\,|\,u_0= 1\}.\end{array}$$
A picture of this Newton subdivision in the case $n=2$ and $d=4$ is given in Figure \ref{fig:linear}.
 Since we chose $f_0$ and $f_1$ to be general, the polynomial $f$ is sch{\"o}n.


 \tdplotsetmaincoords{79}{73}
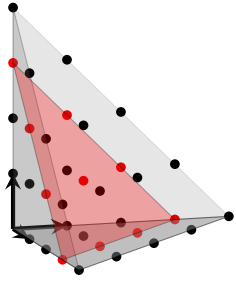
\begin{figure}
  \begin{tikzpicture}[%
    tdplot_main_coords,
    scale=0.75,
    >=stealth
  ]   
\path (0,0,1) node[circle, fill, inner sep=1.3]{};
\path (0,0,2) node[circle, fill, inner sep=1.3]{};
\path (0,0,3) node[circle, fill,color=red, inner sep=1.3]{};
\path (0,0,4) node[circle, fill, inner sep=1.3]{};

\path (0,1,0) node[circle, fill, inner sep=1.3]{};
\path (0,1,1) node[circle, fill, inner sep=1.3]{};
\path (0,1,2) node[circle, fill,color=red, inner sep=1.3]{};
\path (0,1,3) node[circle, fill, inner sep=1.3]{};

\path (0,2,0) node[circle, fill, inner sep=1.3]{};
\path (0,2,1) node[circle, fill,color=red, inner sep=1.3]{};
\path (0,2,2) node[circle, fill, inner sep=1.3]{};

\path (0,3,0) node[circle, fill,color=red, inner sep=1.3]{};
\path (0,3,1) node[circle, fill, inner sep=1.3]{};

\path (0,4,0) node[circle, fill, inner sep=1.3]{};

\path (1,0,0) node[circle, fill, inner sep=1.3]{};
\path (1,0,1) node[circle, fill, inner sep=1.3]{};
\path (1,0,2) node[circle, fill, color=red,inner sep=1.3]{};
\path (1,0,3) node[circle, fill, inner sep=1.3]{};

\path (1,1,0) node[circle, fill, inner sep=1.3]{};
\path (1,1,1) node[circle, fill,color=red, inner sep=1.3]{};
\path (1,1,2) node[circle, fill, inner sep=1.3]{};

\path (1,2,0) node[circle, fill, color=red,inner sep=1.3]{};
\path (1,2,1) node[circle, fill, inner sep=1.3]{};

\path (1,3,0) node[circle, fill, inner sep=1.3]{};

\path (2,0,0) node[circle, fill, inner sep=1.3]{};
\path (2,0,1) node[circle, fill, color=red,inner sep=1.3]{};
\path (2,0,2) node[circle, fill, inner sep=1.3]{};

\path (2,1,0) node[circle, fill, color=red,inner sep=1.3]{};
\path (2,1,1) node[circle, fill, inner sep=1.3]{};

\path (2,2,0) node[circle, fill, inner sep=1.3]{};

\path (3,0,0) node[circle, fill,color=red, inner sep=1.3]{};
\path (3,0,1) node[circle, fill, inner sep=1.3]{};

\path (3,1,0) node[circle, fill, inner sep=1.3]{};

\path (4,0,0) node[circle, fill, inner sep=1.3]{};

    \draw[ultra thick,->] (0,0,0) -- (1,0,0);
    \draw[ultra thick,->] (0,0,0) -- (0,1,0);
    \draw[ultra thick,->] (0,0,0) -- (0,0,1);    
    \draw[fill=gray,opacity=0.5] (0,0,0) -- (0,4,0) -- (4,0,0) -- cycle;
    \draw[fill=gray,opacity=0.25] (0,0,0) -- (0,0,4) -- (4,0,0) -- cycle;
    \draw[fill=black,opacity=0.1] (0,0,0) -- (0,4,0) -- (0,0,4) -- cycle;
    \draw[fill=red,opacity=0.3] (3,0,0) -- (0,3,0) -- (0,0,3) -- cycle;      
  \end{tikzpicture}   
  \caption{Degenerating a quartic surface into a union of a cubic surface and a rational surface, intersecting along a cubic curve.}
\label{fig:linear}
\end{figure}

 The toric $k[t]$-scheme $\mathbb{P}_k(\widetilde{\Delta})$ defined by $(\Delta_f,\varphi_f)$ is the blow-up of
 $$\mathbb{P}^{n+1}_{k[t]}=\mathrm{Proj}\, k[t][z_0,\ldots,z_{n+1}]$$ along the hyperplane $H$ defined by $z_0=0$ in
 the special fiber $\mathbb{P}^{n+1}_k$. The toric $R$-scheme $\cX$ is then given by
 $$\cX=\mathbb{P}_k(\widetilde{\Delta})\times_{k[t]}R.$$
  We write $\cX_k=D_1+D_2$ where $D_1\cong \mathbb{P}^{n+1}_k$ is the strict transform of $\mathbb{P}^{n+1}_k$ and $D_2$ is the exceptional divisor of the blow-up; thus $D_2$ is the projective bundle
  $\mathbb{P}(\mathcal{O}_H\oplus \mathcal{O}_H(1))$ over
  $H$, and it is isomorphic to the toric variety $\mathbb{P}_k(\delta_{\leq})$.
 The scheme $Z(f)$ is the hypersurface in $\mathbb{P}^{n+1}_K$ defined by $f$, and $\cY$ is its schematic closure in $\cX$. The scheme $\cY$ is a strictly semi-stable proper $R$-model for $Z(f)$. Its special fiber consists of two irreducible components, $E_1=\cY\cap D_1$ and $E_2=\cY\cap D_2$. We can explicitly describe the closed strata of the special fiber $\cY_k$ in the following way.
 \begin{itemize}
\item The component $E_1$ is the degree $(d-1)$ hypersurface in $\mathbb{P}^{n+1}_k$ defined by $f_1=0$; it is isomorphic to $Z(\overline{f}_{\delta_{\geq}})$.
\item Let $H^o$ be the open subscheme of $H$ defined by $z_{n+1}\neq 0$. We fix a trivialization of the line bundle $\mathcal{O}_{H^o}(1)$. Then $E_2$ is the schematic closure in $D_2$ of the closed subscheme of $$\mathbb{P}(\mathcal{O}_{H^o}\oplus \mathcal{O}_{H^o}(1))\cong H^o\times_k\mathrm{Proj}\,k[w_0,w_1]$$ defined by
 $$w_0f_0\left(\frac{z_1}{z_{n+1}},\ldots,\frac{z_{n}}{z_{n+1}},1\right)+w_1f_1\left(0,\frac{z_1}{z_{n+1}},\ldots,\frac{z_{n}}{z_{n+1}},1\right)=0.$$
    Thus $E_2$ is rational, and isomorphic to $Z(\overline{f}_{\delta_{\leq}})$.
\item  The intersection $E_1\cap E_2$ is the degree $(d-1)$ hypersurface in $H\cong \mathbb{P}^{n}_k$ defined by $f_1(0,z_1,\ldots,z_{n+1})=0$; it is isomorphic to $Z(\overline{f}_{\delta_{=}})$.
\end{itemize}
Now Theorem \ref{theo:nondeg} simply reproduces the formula \eqref{eq:volsb} for $\Volsb(\sbir{Z(f)})$ in terms of the semi-stable model $\cY$.
\end{exam}

\begin{rema}
Note that, in the set-up of Example \ref{exam:linear}, there is a more straightforward degeneration with similar properties.
 Consider the general homogeneous degree $d$ polynomial $g=f_0+z_0f_1$ in $k[z_0,\ldots,z_{n+1}]$ and 
the $R$-scheme 
$$\cY'=\mathrm{Proj}\,R[z_0,\ldots,z_{n+1}]/(tg-z_0f_1).$$
By Example \ref{exam:toroidal}, this scheme is strictly toroidal, but not strictly semi-stable at the points where $t=g=z_0=f_1=0$. The special fiber is the union of the hyperplane in $\mathbb{P}_k^{n+1}$ defined by $z_0=0$ and the degree $d-1$ projective hypersurface defined by $f_1=0$. The degeneration $\cY$ in Example \ref{exam:linear} is the blow-up of $\cY'$ along the non-Cartier divisor defined by $z_0=t=0$.
 Generally speaking, regular subdivisions of Newton polytopes produce degenerations in large ambient spaces, because the process creates an irreducible divisor in the special fiber of the ambient toric scheme for each face of maximal dimension in the subdivision. 
\end{rema}
 \section{Variation of stable birational type}
In order to verify the non-triviality of the tropical obstruction to stable rationality in Theorem \ref{theo:main}, we need to control possible cancellations between the contributions of different faces of the Newton polytope. For this purpose, we will prove that under suitable conditions one can vary the stable birational type of a very general hypersurface with fixed Newton polytope while preserving a chosen boundary stratum associated with a face of the Newton polytope. This result is interesting in its own right; it generalizes theorems on variation of stable birational types in families due to Shinder \cite{shinder} and Schreieder \cite{schreieder-var}, at least in characteristic zero: Schreieder's results are valid in positive characteristic, as well. But its main relevance for us lies in the the calculations of tropical obstructions to stable rationality: we will give a first application in Theorem \ref{theo:slope}, and further applications in Section \ref{sec:app}.

\subsection{Comparison of faces of Newton polytopes}
\begin{theo}\label{theo:variation}
Let $M$ be a lattice of rank $n+1$.
Let $\Delta$ be a lattice polytope in $M_{\R}$, and assume that $\Delta$ is stably irrational.
  \begin{enumerate}
\item \label{it:varyW} Let $W$ be an integral $k$-scheme of finite type.
 Assume that $\Delta$ admits a regular integral polyhedral subdivision $\mathscr{P}$ such that every face of $\mathscr{P}$ not contained in the boundary of $\Delta$ is stably rational. Then, for every very general polynomial $g$ in $k[M]$ with Newton polytope $\Delta$, the hypersurface $Z^o(g)$ is not stably birational to $W$.
\item \label{it:varyface} Let $\delta$ be a lattice polytope contained in the boundary of $\Delta$. Assume that $\Delta$ admits a regular integral polyhedral subdivision $\mathscr{P}$ satisfying the following properties:
     \begin{enumerate}
    \item the polytope $\delta$ is a face of $\mathscr{P}$;
    \item every face  of $\mathscr{P}$ that intersects $\delta$ and is not contained in the boundary of $\Delta$ is stably rational;
    \item every face $\tau$ of $\mathscr{P}$ that is not contained in the boundary of $\Delta$ admits a regular integral polyhedral subdivision $\mathscr{Q}$ such that every face of $\mathscr{Q}$ not contained in the boundary of $\tau$ is stably rational.
    \end{enumerate}
          Then, for every very general polynomial $g$ in $k[M]$ with Newton polytope $\Delta$, the hypersurface $Z^o(g)$ is not stably birational to $Z^o(g_{\delta})$.
 \end{enumerate}
\end{theo}
\begin{proof}
\eqref{it:varyW} By Proposition \ref{prop:stabyratpol}, it suffices to construct a polynomial $f$ in $K[M]$ with Newton polytope
$\Delta$ such that $f$ is Newton non-degenerate and $Z^o(f)$ is not stably birational to $W\times_k K$. Let $f$ be a sch{\"o}n patchworking polynomial with Newton polytope $\Delta$ and Newton subdivision $\mathscr{P}$ as constructed in Section \ref{ss:viro}, for a very general choice of coefficients $d_m$ in $k$. Then $Z^o(f)$ is stably irrational by our assumption that $\Delta$ is stably irrational. Using Theorem \ref{theo:vol} to compute the stable birational volume of $Z^o(f)$, we find that
$$\Volsb(\sbir{Z^o(f)})=(-1)^{\dim(\Delta)}\sum_{\delta \nsubseteq  \partial \Delta} (-1)^{\dim(\delta)}\sbir{Z^o(\overline{f}_\delta)}=a\sbir{\Spec k}$$
for some integer $a$, because the faces $\delta$ appearing in the sum are stably rational.
 On the other hand, $\Volsb(\sbir{W\times_k K})=\sbir{W}$: by resolution of singularities, we may assume that $W$ is smooth and proper, and then the result follows from Example \ref{exam:goodmodel}.  If $Z^o(f)$ were stably birational to $W\times_k K$, this would imply that
$a=1$ and $\sbir{\Spec k}=\sbir{W}$ so that $W$ would be stably rational, contradicting the fact that $Z^o(f)$ is stably irrational.

\eqref{it:varyface}
 We may assume that $\delta$ is stably irrational, since otherwise, the result is obvious.
 We choose a toric resolution of singularities $\pi\colon X\to \mathbb{P}_{\Q}(\Delta)$ and we consider
 the universal family
$$\theta\colon \mathscr{Y}\times_{\mathbb{P}H^0(\mathbb{P}_{\Q}(\Delta),\mathcal{L}(\Delta))}U\to U$$ of hypersurfaces in $X$ with Newton polytope $\Delta$, using the notations from the proof of Proposition \ref{prop:stabyratpol} (with $F_0=\Q$).
 For every algebraically closed field $F$ of characteristic $0$,
the points $u$ in $U(F)$ correspond canonically to the polynomials $g$ in $F[M]$ with Newton polytope $\Delta$, up to scaling by a factor in $F^{\times}$. Under this correspondence, the fiber $\theta^{-1}(u)$ is the schematic closure of $Z^o(g)$ in $X\times_{\Q}F$.

 We similarly choose a toric resolution $\pi_{\delta} \colon X_{\delta}\to \mathbb{P}_{\Q}(\delta)$ and consider the universal family
 $$\theta_{\delta}\colon \mathscr{Y}_{\delta}\times_{\mathbb{P}H^0(\mathbb{P}_{\Q}(\delta),\mathcal{L}(\delta))}U_{\delta}\to U_{\delta}$$ of hypersurfaces in $X_{\delta}$ with Newton polytope $\delta$. Let
 $U'$ be the dense open subscheme of $U$ that parameterizes polynomials $g$ such that  $g_{\delta}$ has Newton polytope $\delta$. Then the linear projection
 $$\mathbb{P}H^0(\mathbb{P}_{\Q}(\Delta),\mathcal{L}(\Delta))\dashrightarrow \mathbb{P}H^0(\mathbb{P}_{\Q}(\delta),\mathcal{L}(\delta)),\,g\mapsto g_{\delta}$$
 is defined on $U'$ and maps $U'$ into $U_{\delta}$.

 By Bertini's theorem, there is a dense open subscheme $V$ of $U'$ such that the families
 $$\mathscr{Y}\times_{\mathbb{P}H^0(\mathbb{P}_{\Q}(\Delta),\mathcal{L}(\Delta))}V\to V,\qquad \mathscr{Y}_{\delta}\times_{\mathbb{P}H^0(\mathbb{P}_{\Q}(\delta),\mathcal{L}(\delta))}V\to V$$
 are smooth. By Corollary \ref{coro:verygen}, the locus of stably birational geometric fibers in   these families is a countable union of closed subsets of $V$.
   Thus it suffices to construct one Newton non-degenerate polynomial $g$ in $F[M]$ with Newton polytope $\Delta$, for some algebraically closed field $F$ of characteristic $0$, with the following properties:
   \begin{itemize}\item the polynomial $g_{\delta}$ has Newton polytope $\delta$ and is Newton non-degenerate;
   \item the scheme $Z^o(g)$ is not stably birational to $Z^o(g_\delta)$.
   \end{itemize}

Let $f\in K[M]$ be a patchworking polynomial with Newton polytope $\Delta$ and Newton subdivision $\mathscr{P}$ as constructed in Section \ref{ss:viro}, for a very general choice of coefficients $d_m$ in $k$. Then $f_\delta$ has Newton polytope $\delta$, and $f$ and $f_\delta$ are sch{\"o}n; in particular, they are Newton non-degenerate, by Proposition \ref{prop:schon-nondeg}. Moreover, by our assumption that $\delta$ is stably irrational, we know that $Z^o(\overline{f}_\delta)$ is stably irrational.

  By Theorem \ref{theo:main}, we have
  $$\Volsb(\sbir{Z^o(f_\delta)})=\sbir{Z^o(\overline{f}_\delta)}.$$
Thus is suffices to prove that
 $$\Volsb(\sbir{Z^o(f)})\neq \sbir{Z^o(\overline{f}_\delta)}.$$
  Again using Theorem \ref{theo:main} to compute $\Volsb(\sbir{Z^o(f)})$, we see that
   it is enough to show that
      $Z^o(\overline{f}_\delta)$ is not stably birational to $Z^o(\overline{f}_\tau)$ for any face $\tau$ of $\mathscr{P}$ that meets the relative interior of $\Delta$. This certainly holds when $\tau$ is stably rational; thus,
            by our assumptions, we may suppose that $\tau$ is disjoint from $\delta$ and that $Z^o(\overline{f}_\tau)$ is not stably rational. Since the coefficients of $\overline{f}_\tau$ are very general with respect to those of $\overline{f}_\delta$, it follows from point \eqref{it:varyW} that
            $Z^o(\overline{f}_\tau)$ is not stably birational to $Z^o(\overline{f}_\delta)$.
\end{proof}

\begin{coro}\label{coro:shinder}
Let $d$ be a positive integer. Let $W$ be an integral $k$-scheme of finite type. If a very general degree $d$ hypersurface in $\mathbb{P}^{n+1}_k$ is not stably rational, then  a very general degree $d$ hypersurface in $\mathbb{P}^{n+1}_k$ is not stably birational to $W$.
\end{coro}
\begin{proof}
This is a special case of Theorem \ref{theo:variation}\eqref{it:varyW}, taking $d$ times the standard simplex for the polytope $\Delta$. Then $\Delta$ admits a regular subdivision into unimodular simplices; such simplices are stably rational (see Example \ref{exam:prim}).
\end{proof}

Corollary \ref{coro:shinder} is a refinement of the main result of \cite{shinder}, and
implies Theorem 1.1 of \cite{schreieder-var} over fields of characteristic zero. It was proved in \cite{schreieder-var} that, if we add the assumption that $W$ is smooth and projective and does not admit a decomposition of the diagonal, 
Corollary \ref{coro:shinder} remains valid in positive characteristic, where our methods do not apply.

\begin{coro}\label{coro:hyperplane}
Let $d$ and $n$ be positive integers. Let $H$ be a hyperplane in $\mathbb{P}^{n+1}_k$, and let $X$ be a degree $d$ hypersurface in $\mathbb{P}^{n+1}_k$ that is very general with respect to $H$. If $X$ is stably irrational, then $X$ is not stably birational to $X\cap H$.
\end{coro}
\begin{proof}
Let $z_0,\ldots,z_{n+1}$ be homogeneous coordinates on $\mathbb{P}^{n+1}_k$.
We may assume that $H$ is the coordinate hyperplane defined by $z_0=0$. Let $g$ be a very general  homogeneous polynomial of degree $d$ in $k[z_0,\ldots,z_{n+1}]$. Then the Newton polytope of $g$
is given by
$$\Delta=\{(u_0,\ldots,u_{n+1})\in \R_{\geq 0}^{n+2}\,|\,u_0+\ldots+u_{n+1}= d\}.$$
Let $\delta$ be the face of $\Delta$ defined by $u_0=0$.
 We must prove that $Z^o(g)$ is not stably birational to $Z^o(g_{\delta})$. For this purpose, it suffices to construct a regular polyhedral subdivision $\mathscr{P}$ that satisfies the conditions of Theorem \ref{theo:variation}\eqref{it:varyface} (note that $\Delta$ is stably irrational by our assumption that $X$ is stably irrational). We can take for $\mathscr{P}$ the subdivision whose maximal faces are given by
$$\begin{array}{lll}
 \tau_1&=&\{(u_0,\ldots,u_{n+1})\in \Delta\,|\, u_0\leq 1\},
 \\ \tau_2&=& \{(u_0,\ldots,u_{n+1})\in \Delta\,|\,u_0\geq 1\}.
 \end{array}$$
 This is precisely the subdivision from Example \ref{exam:linear}, see Figure \ref{fig:linear}.
 The only face of $\mathscr{P}$ that intersects $\delta$ and that is not contained in the boundary of $\Delta$ is the face $\tau_1$. This face is stably rational, because it has lattice width $1$ (see Example \ref{exam:prim}). The other faces of $\mathscr{P}$ that meet the relative interior of $\Delta$ admit regular subdivisions into unimodular simplices.
  Thus it follows from Theorem \ref{theo:variation}\eqref{it:varyface} that $Z^o(g)$ is not stably birational to $Z^o(g_{\delta})$.
\end{proof}

\subsection{Raising degree and dimension}
Corollary \ref{coro:hyperplane} has the following interesting application.
\begin{theo}\label{theo:slope}
Let $d$ and $n$ be positive integers. Assume that a very general degree $d$ hypersurface
in $\mathbb{P}^{n+1}_k$ is not stably rational. Then for all integers $n'\geq n$ and $d'\geq d+n'-n$, a very general degree $d'$ hypersurface in $\mathbb{P}^{n'+1}_k$ is not stably rational.
\end{theo}
\begin{proof}
By induction, it suffices to prove the theorem for $n'=n$ and $d'=d+1$, and for $n'=n+1$ and $d'=d+1$. Assume that we are in either of these cases.
 Let $\Delta$ be the dilatation with factor $d+1$ of the unimodular $(n'+1)$-dimensional simplex:
$$\Delta=\{u\in \R^{n'+1}_{\geq 0}\,|\,u_1+\ldots+u_{n'+1}\leq d+1\}.$$
Let $\mathscr{P}$ be the regular polyhedral subdivision of $\Delta$ with maximal faces
$$\begin{array}{lll}
 \tau_1&=&\{(u_1,\ldots,u_{n'+1})\in \R^{n'+1}_{\geq 0}\,|\, u_1+\ldots+u_{n'+1}\leq d\},
 \\[1.5ex] \tau_2&=&\{(u_1,\ldots,u_{n'+1})\in \R^{n'+1}_{\geq 0}\,|\, d\leq u_1+\ldots+u_{n'+1}\leq d+1\}.
 \end{array}$$
   This subdivided polytope is isomorphic to the one from Example \ref{exam:linear} (with $n$ replaced by $n'$ and $d$ by $d+1$), see Figure \ref{fig:linear}.
 We denote by $\sigma$ the intersection of $\tau_1$ and $\tau_2$.

Let $g$ be a very general polynomial in $k[x_1,\ldots,x_{n'+1}]$ with Newton polytope $\Delta$. Then $Z(g)$ is a very general hypersurface of degree $d+1$ in $\mathbb{P}^{n'+1}_k$, and $Z(g_{\tau_2})$ is stably rational because $\tau_2$ has lattice width $1$. Moreover, $Z(g_{\tau_1})$ is a very general hypersurface of degree $d$ in $\mathbb{P}^{n'+1}_k$, and $Z(g_{\sigma})$ is the intersection of $Z(g_{\tau_1})$ with the hyperplane at infinity.

 If $n'=n$, then our assumptions imply that $Z(g_{\tau_1})$ is stably irrational, so that it is not stably birational to $Z(g_{\sigma})$ by Corollary \ref{coro:hyperplane}. If $n'=n+1$, then
 $Z(g_{\sigma})$ is stably irrational, so that it is still not stably birational to $Z(g_{\tau_1})$: if $Z(g_{\tau_1})$ is stably rational then this is trivial, and otherwise it follows again from Corollary \ref{coro:hyperplane}.
 Thus for both values of $n'$, Theorem \ref{theo:main} implies that $\Delta$ is stably irrational.
\end{proof}

We will also need the following variant of Theorem \ref{theo:slope} for products of projective spaces.

 \begin{theo}\label{theo:slope2} Let $\ell,\,m,\,d$ and $e$ be positive integers.
 Suppose that a very general divisor of bidegree $(d,e)$ in $\PP_k^{\ell}\times_k \PP_k^m$ is stably irrational. Then for all integers $\ell'\geq \ell$ and $m'\geq m$, and all integers $d'\geq d+\ell'-\ell$ and $e'\geq e+m'-m$, a 
 very general divisor of bidegree $(d',e')$ in $\PP^{\ell'}_k\times_k \PP_k^{m'}$ is stably irrational.
 \end{theo}
 
 \begin{proof}
 The proof is very similar to that of Theorem \ref{theo:slope}, so we only indicate what needs to be modified. 
  By induction, it suffices to prove the case where $\ell'\in \{\ell,\ell+1\}$ and $m'=m$, $d'=d+1$, $e'=e$.
  We set
  $$
 \Delta_1=\{u\in \RR_{\ge 0}^{\ell'}\,|\, u_1+\ldots+u_{\ell'} \leq d+1\}\text{ and }\Delta_2=\{v\in \RR_{\ge 0}^{m}\,|\, v_1+\ldots+v_{m}\leq e \},$$
 and we write $\Delta=\Delta_1\times \Delta_2$. Then $\Delta$ is the Newton polytope of a general hypersurface of bidegree $(d+1,e)$ in $\mathbb{P}^{\ell'}_k\times_k \mathbb{P}^m_k$. 
 We apply the subdivision from the proof of Theorem \ref{theo:slope} to the polytope $\Delta_1$ (with $n'+1$ replaced by $\ell'$).
  Taking the product with $\Delta_2$ we obtain a regular subdivision of $\Delta$ such that the faces that meet the relative interior of $\Delta$ are $\tau_1'=\tau_1\times \Delta_2$, $\tau_2'=\tau_2\times \Delta_2$ and $\sigma'=\sigma \times \Delta_2$.
 The face  $\tau'_2$ has lattice width $1$ and, therefore, it is stably rational. Let $h$ be a very general polynomial over $k$ with Newton polytope $\tau'_1$.
 Then, just like in the proof of Theorem \ref{theo:slope} (where the role of $h$ was played by $g_{\tau_1}$), it is enough to show that 
  $Z(h)$ is not stably birational to $Z(h_{\sigma'})$. 
 
  If $\ell'=\ell$, the assumption in the statement implies that $Z(h)$ is stably irrational, since it is a very general hypersurface of bidegree $(d,e)$ in $\mathbb{P}^{\ell}_k\times_k \mathbb{P}^{m}_k$. If $\ell'=\ell+1$ then we similarly find that $Z(h_{\sigma'})$ is stably irrational. Thus in both cases, we may assume that $\tau'_1$ is stably irrational. We consider a further regular subdivision of $\tau'_1$ with maximal faces 
 $$\begin{array}{lll}
     \rho_1 &=& \{(u,v)\in \tau'_1\,|\,u_1+\ldots+u_{\ell'}\leq d-1\},
     \\[1.5ex]
      \rho_2&=& \{(u,v)\in \tau'_1\,|\,d-1\leq u_1+\ldots+u_{\ell'}\leq d\}.
 \end{array}$$
The face $\rho_2$ has lattice width $1$ and thus is stably rational. The face $\rho_1$ does not intersect $\sigma$ and admits a regular subdivision into unimodular simplices. Now it follows from Theorem \ref{theo:variation} that $Z(h)$ is not stably birational to $Z(h_{\sigma'})$. 
\end{proof}

\begin{rema}
It may be hard to determine whether a given polytope admits a subdivision into  stably rational polytopes. For instance, consider the simplicial polytope $\Delta$ given as the convex hull of $(6,14,17,65)$ and the four standard basis vectors in $\RR^4$ (this polytope appears in \cite{HZ}). Then the only lattice points contained in $\Delta$ are its vertices, so that $\Delta$ cannot be subdivided further. Yet $\Delta$ is itself not stably rational: its so-called {\em Fine interior} is non-empty, which implies that a generic hypersurface with Newton polytope $\Delta$ has non-negative Kodaira dimension by the results in Jonathan Fine's PhD thesis (Warwick, 1983). See the appendix to \S4 in \cite{reid} and  Theorem 2.18 in \cite{batyrev} for published accounts.
\end{rema}

\section{Hypersurfaces in projective space}\label{sec:app}
\subsection{The quartic fivefold}\label{quarticfivefoldsection}
It has been conjectured that smooth quartic hypersurfaces over $k$ are stably irrational in all dimensions $n>0$. For curves and surfaces this follows from existence of global sections of the canonical bundle. Stable irrationality of a very general quartic threefold was proved by Colliot-Th\'el\`ene and Pirutka in \cite{CTP}, and the fact that a very general quartic fourfold is stably irrational is a special case of Totaro's results in \cite{totaro}. Totaro's results were later improved by Schreieder in \cite{schreieder}. The first open case is that of quartic fivefolds, which are known to be unirational \cite{CM}. We will now prove that a very general quartic fivefold is stably irrational. In fact, we will prove a conditional result in arbitrary dimension: stable irrationality of a ``special'' quartic $(n-1)$-fold implies stable irrationality of a very general degree $d$ hypersurface in $\mathbb{P}^{n+1}_k$ for all $d\geq 4$.

\begin{theo}\label{theo:quartic1}
Let $n$ and $d$ be integers satisfying $n\geq 2$ and $d\geq 4$.
Assume that there exists a stably irrational quartic double $(n-1)$-fold over $k$ with at most isolated ordinary double points as singularities. Then a
 very general hypersurface of degree $d\geq 4$ in $\mathbb{P}^{n+1}_{k}$ is not stably rational.
\end{theo}
\begin{proof}
By Theorem \ref{theo:slope}, it suffices to prove the case $d=4$.
 Let $\Delta$ be the convex hull in $\R^{n+2}$ of the points $4e_0,4e_1,\ldots,\,4e_{n+1}$ where  $(e_0,\ldots,e_{n+1})$ is the standard basis of $\R^{n+2}$. Then $\Delta$ is the Newton polytope of a general quartic hypersurface in $\mathbb{P}^{n+1}_{k}$. We will show that $\Delta$ is stably irrational.

We define a regular integral polyhedral subdivision $\mathscr{P}$ of $\Delta$ by slicing $\Delta$ with the hyperplane
$$
H=\{(u_0,u_1,\ldots,u_{n+1})\in \R^{n+1}\,|\,u_0=u_{n+1}\}.
$$
 A picture of this subdivision for $n=2$ is shown in Figure \ref{fig:quarticdouble}.


 \tdplotsetmaincoords{82}{75}
 \begin{center}
\begin{figure}[h!]
  \begin{tikzpicture}[%
    tdplot_main_coords,
    scale=0.75,
    >=stealth
  ]   
\path (0,0,1) node[circle, fill, inner sep=1.3]{};
\path (0,0,2) node[circle, fill,color=red, inner sep=1.3]{};
\path (0,0,3) node[circle, fill, inner sep=1.3]{};
\path (0,0,4) node[circle, fill, inner sep=1.3]{};

\path (0,1,0) node[circle, fill, inner sep=1.3]{};
\path (0,1,1) node[circle, fill, inner sep=1.3]{};
\path (0,1,2) node[circle, fill, inner sep=1.3]{};
\path (0,1,3) node[circle, fill, inner sep=1.3]{};

\path (0,2,0) node[circle, fill, inner sep=1.3]{};
\path (0,2,1) node[circle, fill, color=red,inner sep=1.3]{};
\path (0,2,2) node[circle, fill, inner sep=1.3]{};

\path (0,3,0) node[circle, fill, inner sep=1.3]{};
\path (0,3,1) node[circle, fill, inner sep=1.3]{};

\path (0,4,0) node[circle, fill, color=red,inner sep=1.3]{};

\path (1,0,0) node[circle, fill, inner sep=1.3]{};
\path (1,0,1) node[circle, fill, inner sep=1.3]{};
\path (1,0,2) node[circle, fill, inner sep=1.3]{};
\path (1,0,3) node[circle, fill, inner sep=1.3]{};

\path (1,1,0) node[circle, fill, inner sep=1.3]{};
\path (1,1,1) node[circle, fill, color=red,inner sep=1.3]{};
\path (1,1,2) node[circle, fill, inner sep=1.3]{};

\path (1,2,0) node[circle, fill, inner sep=1.3]{};
\path (1,2,1) node[circle, fill, inner sep=1.3]{};

\path (1,3,0) node[circle, fill,color=red, inner sep=1.3]{};

\path (2,0,0) node[circle, fill, inner sep=1.3]{};
\path (2,0,1) node[circle, fill,color=red, inner sep=1.3]{};
\path (2,0,2) node[circle, fill, inner sep=1.3]{};

\path (2,1,0) node[circle, fill, inner sep=1.3]{};
\path (2,1,1) node[circle, fill, inner sep=1.3]{};

\path (2,2,0) node[circle, fill,color=red, inner sep=1.3]{};

\path (3,0,0) node[circle, fill, inner sep=1.3]{};
\path (3,0,1) node[circle, fill, inner sep=1.3]{};

\path (3,1,0) node[circle, fill, color=red,inner sep=1.3]{};

\path (4,0,0) node[circle, fill, color=red,inner sep=1.3]{};

    \draw[ultra thick,->] (0,0,0) -- (1,0,0);
    \draw[ultra thick,->] (0,0,0) -- (0,1,0);
    \draw[ultra thick,->] (0,0,0) -- (0,0,1);    
    \draw[fill=gray,opacity=0.5] (0,0,0) -- (0,4,0) -- (4,0,0) -- cycle;
    \draw[fill=gray,opacity=0.25] (0,0,0) -- (0,0,4) -- (4,0,0) -- cycle;
    \draw[fill=black,opacity=0.1] (0,0,0) -- (0,4,0) -- (0,0,4) -- cycle;
    \draw[fill=red,opacity=0.3] (4,0,0) -- (0,4,0) -- (0,0,2) -- cycle;      
       \path (0,0,4) node[circle, fill, inner sep=1.3]{};
        \path (0,0,3) node[circle, fill, inner sep=1.3]{};
  \end{tikzpicture}   
   \caption{Degenerating a quartic surface into a union of two isomorphic quartic double surfaces  intersecting along a quartic double curve.}
 \label{fig:quarticdouble}
\end{figure}
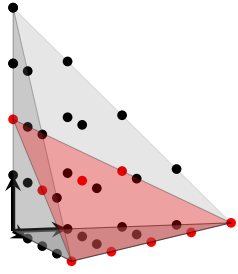
\end{center}

 Then $\mathscr{P}$ has precisely three faces that are not contained in the boundary of $\Delta$: two faces of dimension $n+1$, namely,
 $$\begin{array}{lcl}
 \delta_{\leq}&=&\{(u_0,u_1,\ldots,u_{n+1})\in \Delta\,|\,u_0\leq u_{n+1}\}
 \\  \delta_{\geq}&=&\{(u_0,u_1,\ldots,u_{n+1})\in \Delta\,|\,u_0\geq u_{n+1}\}
 \end{array}$$
and one face of dimension $n$, namely,
 $$ \delta_{=}=\{(u_0,u_1,\ldots,u_{n+1})\in \Delta\,|\,u_0= u_{n+1}\}.$$
The unimodular involution
$$\tau\colon \R^{n+2}\to \R^{n+2},(u_0,\ldots,u_{n+1})\mapsto (u_{n+1},u_1,\ldots,u_n,u_0)$$
that swaps the first and last coordinate preserves $\Delta$ and $\delta_{=}$, and exchanges $\delta_{\leq}$ and $\delta_{\geq}$.
 The subdivision $\mathscr{P}$ is regular: it is induced by the convex piecewise linear function
 $$\varphi\colon \Delta\to \R,\,(u_0,\ldots,u_n)\mapsto \max\{u_0,u_{n+1}\}.$$

 We make a very general choice of coefficients $d_m$ in $k$, where $m$ runs over the lattice points in $\delta_{\leq}$. For every lattice point $m$ in $\delta_{\geq}$, we set $d_m=d_{\tau(m)}$.
 Then the data $(\Delta,\mathscr{P},d_m)$ satisfy the non-degeneracy conditions of Theorem \ref{theo:main}. For every face $\delta$ of $\Delta$, we define the polynomial $g_{\delta}$ as in Theorem \ref{theo:main}:
 $$g_{\delta}=\sum_{m\in \delta\cap \Z^{n+2}}d_mx^m\in k[\Z^{n+2}].$$

 Now, the face $\delta_{=}$ is the Newton polytope of a general quartic double $(n-1)$-fold.
   By Theorem 4.3.1 in \cite{NiSh}, our assumption implies that a very general quartic double $(n-1)$-fold is stably irrational.
  Thus it follows from Proposition \ref{prop:stabyratpol} that $Z^o(g_{\delta_{=}})$ is stably irrational. Moreover, by the symmetry in our choice of the coefficients $d_m$, the schemes $Z^o(g_{\delta_{\leq}})$ and $Z^o(g_{\delta_{\geq}})$ are isomorphic. It follows that
  $$\sbir{Z^o(g_{\delta_{\leq}})}+\sbir{Z^o(g_{\delta_{\geq}})}- \sbir{Z^o(g_{\delta_{=}})}=2\sbir{Z^o(g_{\delta_{\leq}})}-\sbir{Z^o(g_{\delta_{=}})}\neq \sbir{\Spec k}$$ in $\Z[\SB_k]$, because $\sbir{Z^o(g_{\delta_{=}})}\neq \sbir{\Spec k}$. Now Theorem \ref{theo:main} implies that $\Delta$ is stably irrational.
\end{proof}

\begin{coro}\label{cor:quartic-fivefold}
A very general hypersurface of degree $d\geq 4$ in $\mathbb{P}^{5}_k$ or $\mathbb{P}^{6}_k$ is not stably rational. In particular, a very general quartic fivefold is not stably rational.
\end{coro}
\begin{proof}
The case of dimension $4$ follows from Artin and Mumford's famous example of a stably irrational quartic double solid with isolated ordinary double point singularities \cite{Artin-Mumford}.
 The case of dimension $5$ follows from the result by Hassett, Pirutka \& Tschinkel that
 a very general quartic double fourfold is stably irrational \cite{HPTdouble}.
\end{proof}

The only new case covered by Corollary \ref{cor:quartic-fivefold} is the quartic fivefold: all other cases of the corollary were already contained in the range of results by Schreieder \cite{schreieder}, so we merely obtain an alternative proof.
 
 As a further illustration of our techniques, we also present the following variant of Theorem \ref{theo:quartic1}.

\begin{theo}\label{theo:quartic2}
Let $\ell$ and $d$ be integers satisfying $\ell\geq 1$ and  $d\geq 4$.
Assume that there exists a smooth hypersurface of bidegree $(2,2)$ in $\mathbb{P}^\ell_{k}\times_k \mathbb{P}_{k}^{\ell}$ (resp.~$\mathbb{P}^\ell_{k}\times_k \mathbb{P}_{k}^{\ell+1}$) which is stably irrational.
Then  for $n=2\ell$ (resp.~$n=2\ell+1$), a very general hypersurface of degree $d$ in $\mathbb{P}^{n+1}_{k}$ is not stably rational.
\end{theo}

\begin{proof}
 By Theorem \ref{theo:slope}, we may assume that $d=4$.
 We will prove the case $n=2\ell$; the case $n=2\ell+1$ is entirely analogous.
 Let $\Delta$ be the convex hull in $\R^{n+2}$ of the points $4e_0,\ldots,\,4e_{n+1}$ where  $(e_0,\ldots,e_{n+1})$ is the standard basis of $\R^{n+2}$. Then $\Delta$ is the Newton polytope of a general quartic hypersurface in $\mathbb{P}^{n+1}_{k}$. We define a regular integral polyhedral subdivision $\mathscr{P}$ of $\Delta$ by slicing $\Delta$ with the hyperplanes
$$
H_a=\{(u_0,\ldots,u_{n+1})\in \R^{n+1}\,|\,u_{0}+\ldots+u_{\ell}=a\}
$$
for $a=1,2,3$. The face $H_a\cap \Delta$ of $\mathscr{P}$ will be denoted by $\delta_a$. A picture of this subdivision for $\ell=1$ and $n=2\ell=2$  is given in Figure \ref{fig:bideg2}.

\begin{center}
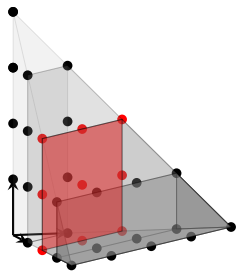
\begin{figure}[h!]
 \tdplotsetmaincoords{82}{75}
  \begin{tikzpicture}[%
    tdplot_main_coords,
    scale=0.75,
    >=stealth
  ]   
      \draw[color=gray] (3,0,0) -- (0,3,0);
            \draw[color=gray] (1,0,0) -- (0,1,0);
            
\path (0,0,1) node[circle, fill, inner sep=1.3]{};
\path (0,0,2) node[circle, fill, inner sep=1.3]{};
\path (0,0,3) node[circle, fill, inner sep=1.3]{};
\path (0,0,4) node[circle, fill, inner sep=1.3]{};

\path (0,1,0) node[circle, fill, inner sep=1.3]{};
\path (0,1,1) node[circle, fill, inner sep=1.3]{};
\path (0,1,2) node[circle, fill, inner sep=1.3]{};
\path (0,1,3) node[circle, fill, inner sep=1.3]{};

\path (0,2,0) node[circle, fill, color=red,inner sep=1.3]{};
\path (0,2,1) node[circle, fill, color=red,inner sep=1.3]{};
\path (0,2,2) node[circle, fill, color=red,inner sep=1.3]{};

\path (0,3,0) node[circle, fill, inner sep=1.3]{};
\path (0,3,1) node[circle, fill, inner sep=1.3]{};

\path (0,4,0) node[circle, fill, inner sep=1.3]{};

\path (1,0,0) node[circle, fill, inner sep=1.3]{};
\path (1,0,1) node[circle, fill, inner sep=1.3]{};
\path (1,0,2) node[circle, fill, inner sep=1.3]{};
\path (1,0,3) node[circle, fill, inner sep=1.3]{};

\path (1,1,0) node[circle, fill, color=red,inner sep=1.3]{};
\path (1,1,1) node[circle, fill, color=red,inner sep=1.3]{};
\path (1,1,2) node[circle, fill, color=red,inner sep=1.3]{};

\path (1,2,0) node[circle, fill, inner sep=1.3]{};
\path (1,2,1) node[circle, fill, inner sep=1.3]{};

\path (1,3,0) node[circle, fill, inner sep=1.3]{};

\path (2,0,0) node[circle, fill,color=red, inner sep=1.3]{};
\path (2,0,1) node[circle, fill, color=red,inner sep=1.3]{};
\path (2,0,2) node[circle, fill, color=red,inner sep=1.3]{};

\path (2,1,0) node[circle, fill, inner sep=1.3]{};
\path (2,1,1) node[circle, fill, inner sep=1.3]{};

\path (2,2,0) node[circle, fill, inner sep=1.3]{};

\path (3,0,0) node[circle, fill, inner sep=1.3]{};
\path (3,0,1) node[circle, fill, inner sep=1.3]{};

\path (3,1,0) node[circle, fill, inner sep=1.3]{};

\path (4,0,0) node[circle, fill, inner sep=1.3]{};

    \draw[ thick,->] (0,0,0) -- (1,0,0);
    \draw[ thick,->] (0,0,0) -- (0,1,0);
    \draw[ thick,->] (0,0,0) -- (0,0,1);    
        
   \draw[opacity=0.1] (4,0,0) -- (0,0,4) -- cycle;
       \draw[fill=gray,opacity=0.50] (0,4,0) -- (0,3,0) -- (0,3,1) -- cycle;
       \draw[fill=gray,opacity=0.50] (4,0,0) -- (3,0,0) -- (3,0,1) -- cycle;
       
    \draw[fill=gray,opacity=0.1] (0,0,0) -- (0,4,0) -- (0,0,4) -- cycle;

       \draw[fill=gray,opacity=0.1] (4,0,0) -- (1,0,0) -- (1,0,3) --(0,1,3) -- (0,1,0) -- (0,4,0) -- cycle;
        \draw[fill=gray,opacity=0.15] (4,0,0) -- (1,0,0) -- (1,0,3) --(0,1,3) -- (0,4,0) -- cycle;
        
            \draw[fill=gray,opacity=0.2] (4,0,0) -- (2,0,0) -- (2,0,2) --(0,2,2) -- (0,2,0) -- (0,4,0) -- cycle;
        \draw[fill=gray,opacity=0.20] (4,0,0) -- (2,0,0) -- (2,0,2) --(0,2,2) -- (0,4,0) -- cycle;

            \draw[fill=gray,opacity=0.20] (4,0,0) -- (3,0,0) -- (3,0,1) --(0,3,1) -- (0,3,0) -- (0,4,0) -- cycle;
        \draw[fill=gray,opacity=0.30] (4,0,0) -- (3,0,0) -- (3,0,1) --(0,3,1) -- (0,4,0) -- cycle;

            \draw[fill=red,opacity=0.4] (2,0,0) -- (0,2,0) -- (0,2,2) -- (2,0,2) -- cycle;      
       \path (0,0,4) node[circle, fill, inner sep=1.3]{};
        \path (0,0,3) node[circle, fill, inner sep=1.3]{};
  \end{tikzpicture}   
\caption{Degenerating a quartic surface to a chain of four rational varieties, two of which intersect along a bidegree $(2,2)$ curve.}
\label{fig:bideg2}
\end{figure}
\end{center}

 Every face $\delta\neq \delta_2$ that is not contained in the boundary of $\Delta$ has lattice width $1$; thus these faces are stably rational, by Example \ref{exam:prim}.
 The face $\delta_2$
 is the Newton polytope of a general hypersurface of bidegree $(2,2)$ in $\mathbb{P}^{\ell}_{k}\times_{k} \mathbb{P}^{\ell}_{k}$; thus it is stably irrational by the assumption in the statement. 
 Now it follows from Corollary \ref{coro:irratpol} that $\Delta$ is stably irrational, which means that a very general quartic of dimension $n$ is not stably rational.
\end{proof}

\begin{rema}
We can simplify the proof of Theorem \ref{theo:quartic2} in the case $n=2\ell$ by using a symmetry argument as in the proof of Theorem \ref{theo:quartic1}, only slicing $\Delta$ by $H_2$ and exploiting the unimodular involution
$$(u_0,\ldots,u_{n+1})\mapsto (u_{\ell+1},\ldots,u_{n+1},u_0,\ldots,u_\ell)$$ that preserves $\Delta$ and exchanges the two faces of maximal dimension in the subdivision.
However, this symmetry argument does not generalize to the case $n=2\ell+1$, which is why we have given the more general proof here.
\end{rema}

Theorem \ref{theo:quartic2} gives another proof of the result that a very general quartic fivefold is stably irrational: it was shown in \cite{HPTbideg2} that a very general bidegree $(2,2)$ hypersurface in $\mathbb{P}^2_k\times_k \mathbb{P}^3_k$ is stably irrational.

\subsection{Results in higher dimensions}
 The tropical degenerations appearing in the proofs of Theorems \ref{theo:quartic1} and \ref{theo:quartic2} in the case $d=4$ are  simple enough to write down explicit equations; see Example 4.3.2 in \cite{NO} for the case of  Theorem \ref{theo:quartic1}.
Using more sophisticated subdivisions, we can extend the proofs to give conditional results for hypersurfaces of higher degrees and dimensions (Proposition \ref{doublehypersurface}). 
 Here and further on in the paper, our tropical methods really pay off, as it would be extremely cumbersome to construct these degenerations by hand.

As input for our tropical degeneration, we first construct some suitable stably irrational lattice polytopes.

\begin{lemm}\label{subdivlemma2}
Let $\ell$ be a positive integer and let $d=2m$ be an even positive integer. Suppose that there exists a stably irrational double cover of $\PP^\ell_k$ branched along a smooth degree $d$ hypersurface.
Let $j$ be a positive integer such that $j\le m$. Let $a_1,\ldots,a_j$ be positive integers such that $a_1+\ldots+a_j=m$. Then the following lattice polytope is stably irrational:
$$
\Delta=\left\{(u,v)\in \R^{\ell+j}_{\geq 0}\,\,\big|\,\,\sum_{i=1}^\ell u_i +a_1v_1+\ldots+a_jv_j \le d  ,\, v_1,\ldots, v_j\in [0,2]\right\}.
$$
\end{lemm}

 \tdplotsetmaincoords{75}{125}
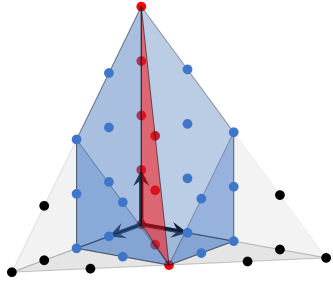
\begin{figure}[h!]
    \centering
  \begin{tikzpicture}[%
    tdplot_main_coords,
    scale=0.75,
    >=stealth
  ]   
  \path (0,0,0) node[circle, color=red, fill, inner sep=1.3]{};
\path (0,0,1) node[circle, color=red, fill, inner sep=1.3]{};
\path (0,0,2) node[circle, color=red, fill, inner sep=1.3]{};
\path (0,0,3) node[circle, color=red,fill, inner sep=1.3]{};
\path (0,0,4) node[circle, color=red, fill, inner sep=1.3]{};
\path (1,1,0) node[circle, fill,  color=red,inner sep=1.3]{};
\path (1,1,1) node[circle, fill,  color=red,inner sep=1.3]{};
\path (1,1,2) node[circle, fill,  color=red,inner sep=1.3]{};
\path (2,2,0) node[circle, fill, color=red, inner sep=1.3]{};

      \draw[ultra thick,->] (0,0,0) -- (1,0,0);
    \draw[ultra thick,->] (0,0,0) -- (0,1,0);
    \draw[ultra thick,->] (0,0,0) -- (0,0,1);    
    \draw[fill=gray,opacity=0.2] (0,0,0) -- (0,4,0) -- (4,0,0) -- cycle;
    \draw[fill=gray,opacity=0.1] (0,0,0) -- (0,0,4) -- (4,0,0) -- cycle;
    \draw[fill=black,opacity=0.05] (0,0,0) -- (0,4,0) -- (0,0,4) -- cycle;

         \draw[fill=mycol,opacity=0.25] (2,0,0) -- (2,0,2) -- (0,0,4) -- (0,0,0) --cycle;   
    \draw[fill=mycol,opacity=0.4] (2,0,0) -- (2,2,0) -- (0,2,0) -- (0,0,0) --cycle;      

   \draw[fill=mycol,opacity=0.3] (2,0,0)  -- (2,0,2) -- (2,2,0) --cycle;        
                \draw[fill=mycol,opacity=0.3] (2,2,0)  -- (0,2,2) -- (0,2,0) --cycle;

 
      \draw[fill=mycol,opacity=0.01] (0,0,4)-- (2,0,2)  -- (2,2,0) -- (0,2,2) --cycle;     
       \draw[fill=mycol,opacity=0.2] (0,0,4)-- (0,0,0)  -- (2,0,0) -- (2,0,2) --cycle;      
         \draw[fill=mycol,opacity=0.3] (0,0,0)-- (0,0,4)  -- (0,2,2) -- (0,2,0) --cycle;     

\path (0,1,0) node[circle, color=mycol, fill, inner sep=1.3]{};
\path (0,1,1) node[circle,  color=mycol,fill, inner sep=1.3]{};
\path (0,1,2) node[circle,  color=mycol,fill, inner sep=1.3]{};
\path (0,1,3) node[circle, color=mycol, fill, inner sep=1.3]{};

\path (0,2,0) node[circle, color=mycol, fill, inner sep=1.3]{};
\path (0,2,1) node[circle,  color=mycol,fill, inner sep=1.3]{};
\path (0,2,2) node[circle, color=mycol, fill, inner sep=1.3]{};

\path (0,3,0) node[circle, fill, inner sep=1.3]{};
\path (0,3,1) node[circle, fill, inner sep=1.3]{};

\path (0,4,0) node[circle, fill, inner sep=1.3]{};

\path (1,0,0) node[circle, color=mycol, fill, inner sep=1.3]{};
\path (1,0,1) node[circle,  color=mycol,fill, inner sep=1.3]{};
\path (1,0,2) node[circle,  color=mycol,fill, inner sep=1.3]{};
\path (1,0,3) node[circle, color=mycol, fill, inner sep=1.3]{};

\path (1,2,0) node[circle, fill, color=mycol, inner sep=1.3]{};
\path (1,2,1) node[circle, fill, color=mycol, inner sep=1.3]{};

\path (1,3,0) node[circle, fill, inner sep=1.3]{};

\path (2,0,0) node[circle, fill,  color=mycol,inner sep=1.3]{};
\path (2,0,1) node[circle, fill,  color=mycol,inner sep=1.3]{};
\path (2,0,2) node[circle, fill,  color=mycol,inner sep=1.3]{};

\path (2,1,0) node[circle, fill,  color=mycol,inner sep=1.3]{};
\path (2,1,1) node[circle, fill,  color=mycol,inner sep=1.3]{};

\path (3,0,0) node[circle, fill, inner sep=1.3]{};
\path (3,0,1) node[circle, fill, inner sep=1.3]{};

\path (3,1,0) node[circle, fill, inner sep=1.3]{};

\path (4,0,0) node[circle, fill, inner sep=1.3]{};

    \draw[fill=red,opacity=0.5] (0,0,0) -- (2,2,0) -- (0,0,4) -- cycle;
  \end{tikzpicture}   
    \caption{The polytope $\Delta$ and its stably irrational subpolytope}
    \label{fig:my_label}
\end{figure}

\begin{proof}
Let us first show that $\Delta$ is in fact a lattice polytope. Consider the projection $p\colon \Delta\to [0,2]^j$ onto the the $(v_1,\ldots,v_j)$-plane. Given $(v_1,\ldots,v_j)\in [0,2]^j$, we have $a_1v_1+\ldots+a_jv_j\le 2a_1+\ldots+2a_j= 2m=d$. This implies that $p$ is surjective and if $(u,v)\in \Delta$ is a vertex, we must have that $v=p(u,v)$ is a vertex of $[0,2]^j$, hence a lattice point. However, for a lattice point $v\in [0,2]^j$, the preimage $p^{-1}(v)$ is a lattice polytope (a dilated simplex). Hence every vertex $(u,v)$ of $\Delta$ is a lattice point, and $\Delta$ is a lattice polytope.

We now prove the claim in the lemma by induction on $j$. If $j=1$, then $a_1=m$ and $\Delta$ is exactly the Newton polytope of a  double cover of $\PP^\ell_k$
branched along a general degree $d$ hypersurface, which is stably irrational by our assumption and Corollary \ref{coro:verygen}. Thus we may assume that $j>1$ and that the result holds for strictly smaller values of $j$. By Corollary \ref{coro:irratpol}, it suffices to construct a regular polyhedral subdivision $\mathscr{P}$ of $\Delta$ such that exactly one face of $\mathscr{P}$ that meets the relative interior of $\Delta$ is stably irrational.

 Let $Q=\{0\le v_{j-1},v_{j} \le 2 \}\subset \RR^2$, and let $\pi\colon \Delta\to Q$ be the projection map. Consider the  regular subdivision  of $Q$ with four maximal polytopes $Q_0,Q_1,Q_2,Q_3$, where $$Q_s=\{v_{j}+s-2\le v_{j-1}\le v_{j}+s-1\}$$ for $s=0,1,2,3$. This subdivision is shown in Figure \ref{fig:Qsubdiv}.
  \begin{figure}[h!]
  \begin{tikzpicture}[scale=0.6]
  \definecolor{mycolor}{RGB}{240,240,240}
          \draw[ thick,fill=mycolor] (0,0) -- (0,4) -- (4,4) -- (4,0) -- cycle;
\path (0,0) node[circle, fill, inner sep=1.5]{};
\path (0,2) node[circle, fill, inner sep=1.5]{};
\path (0,4) node[circle, fill, inner sep=1.5]{};

\path (2,0) node[circle, fill, inner sep=1.5]{};
\path (2,2) node[circle, fill, inner sep=1.5]{};
\path (2,4) node[circle, fill, inner sep=1.5]{};

\path (4,0) node[circle, fill, inner sep=1.5]{};
\path (4,2) node[circle, fill, inner sep=1.5]{};
\path (4,4) node[circle, fill, inner sep=1.5]{};
   
    \draw[ thick,-] (0,0) -- (4,0) -- (4,4) --(0,4) -- cycle;
    \draw[ thick,-] (2,0) -- (4,2) -- cycle;
    \draw[ thick,-] (0,2) -- (2,4) -- cycle;
    \draw[ thick,-] (0,0) -- (4,4) -- cycle;
\node at (3.5,0.66) {$Q_0$};
\node at (2.5,1.5) {$Q_1$};
\node at (1.5,2.5) {$Q_2$};
\node at (0.66,3.5) {$Q_3$};
  \end{tikzpicture}\caption{The subdivision of $Q$}
   \label{fig:Qsubdiv}
  \end{figure}
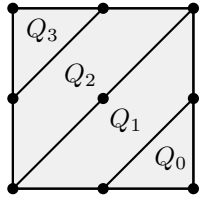
  
We set $Q_{s,s+1}=Q_s\cap Q_{s+1}$ for $s\in \{0,1,2\}$.  The faces not contained in the boundary of $Q$ are the following:

\begin{itemize}
\item codim 0: $Q_0,Q_1,Q_2,Q_3$. All have lattice width one.
\item codim 1: $Q_{01},Q_{12},Q_{23}$. All have lattice width one, except $Q_{12}=\{v_{j-1}=v_j\}\cap Q$.
\end{itemize}Taking preimages under the map $\pi$, this induces a regular subdivision $\mathscr{P}$ of $\Delta$ such that each face that meets the relative interior of $\Delta$ has lattice width one, except for  $\{v_{j-1}=v_{j}\}\cap \Delta$. The latter polytope is isomorphic to
$$
\left\{\sum_{i=1}^\ell u_i +a_1v_1+\ldots+(a_{j-1}+a_j)v_{j-1} \le d \, , v_1,\ldots, v_{j-1}\in [0,2]\right\}\subset \RR^{\ell+j-1}_{\ge 0}.
$$By the induction hypothesis, this polytope is stably irrational, so the proof is complete.
\end{proof}

\begin{lemm}\label{subdivlemma1}
Let $\ell$ be a positive integer and let $d=2m$ be an even positive integer. Suppose that there exists a stably irrational double cover of $\PP^\ell_k$ branched along a smooth degree $d$ hypersurface. Then for any $0\le j\le m$, the following polytope is stably irrational:
$$
\Gamma_j=\left\{(u,v)\in \R_{\geq 0}^{\ell+m}\,|\,\sum_{i=1}^\ell u_i +v_1+\ldots+v_m \le d \, ,\, v_1,\ldots, v_j\in [0,2]\right\}.
$$
\end{lemm}

\begin{proof}
We prove this using downward induction on $j$.  The base case $j=m$ follows from Lemma \ref{subdivlemma2} above. So assume that $j\leq m-1$ and that the result holds for all strictly larger values of $j$.

Consider the regular polyhedral subdivision  $\mathscr P$ of $\Gamma_{j}$ with maximal faces 
\begin{itemize}
    \item $\gamma_1=\Gamma_{j}\cap \{v_j\le 2\}$, which equals $\Gamma_{j+1}$;
    \item $\gamma_2=\Gamma_{j}\cap \{v_j\ge 2\}$.
    \end{itemize}
The faces of $\mathscr{P}$ that meet the relative interior of $\Gamma_j$ are $\gamma_1$, $\gamma_2$ and $\gamma_{12}=\gamma_1\cap \gamma_2$. By the induction hypothesis, the face $\gamma_1=\Gamma_{j+1}$ is stably irrational. 
 We will use Theorem \ref{theo:variation} to argue that a very general hypersurface with Newton polytope $\gamma_{1}$  is not stably birational to a very general hypersurface with Newton polytope $\gamma_{12}$. This in turn implies that $\Gamma_{j}$ is stably irrational by Theorem \ref{theo:main}.

 We consider the subdivision  of $\gamma_1$ with maximal faces $\delta_1=\gamma_1\cap \{v_j\le 1\}$ and $\delta_2=\gamma_1\cap \{v_j\ge 1\}$. These have lattice width one, and therefore are stably rational. Their intersection $\delta_1\cap \delta_2=\gamma_1\cap \{v_j=1\}$ admits a unimodular regular subdivision. Hence the conditions of Theorem \ref{theo:variation}\eqref{it:varyface} are satisfied, which allows us to conclude that $\Gamma_{j}$ is stably irrational. \end{proof}

 \tdplotsetmaincoords{75}{135}
\begin{center}
\begin{figure}[ht]
\centering
   \begin{tikzpicture}[%
    tdplot_main_coords,
    scale=0.75,
    >=stealth
  ]   
 
    \draw[fill=gray,opacity=0.3] (2,0,0) -- (2,2,0) -- (4,0,0) -- cycle;
    \draw[fill=gray,opacity=0.07] (0,0,0) -- (0,0,4) -- (4,0,0) -- cycle;

     \draw[fill=gray,opacity=0.3] (0,4,0) -- (0,0,0) -- (4,0,0) -- cycle;

    \draw[fill=mycol,opacity=0.4] (2,0,0) -- (2,2,0) -- (0,2,0) -- (0,0,0) --cycle;      
   \draw[fill=mycol,opacity=0.45] (2,0,0)  -- (2,0,2) -- (2,2,0) --cycle;        
         
      \draw[fill=mycol,opacity=0.2] (0,0,4)-- (2,0,2)  -- (2,2,0) -- (0,2,2) --cycle;     
       \draw[fill=mycol,opacity=0.2] (0,0,4)-- (0,0,0)  -- (2,0,0) -- (2,0,2) --cycle;      
         \draw[fill=mycol,opacity=0.35] (0,0,0)-- (0,0,4)  -- (0,2,2) -- (0,2,0) --cycle;     
                
\draw[fill=lime!70!black,opacity=0.5] (2,2,0)  -- (0,2,0) -- (0,4,0) --cycle; 
\draw[fill=lime!90!black,opacity=0.3] (0,2,0)  -- (0,2,2) -- (0,4,0) --cycle; 
\draw[fill=lime!90!black,opacity=0.3] (0,2,0)  -- (0,2,2) -- (0,4,0) --cycle; 
\draw[fill=lime!90!black,opacity=0.6] (2,2,0)  -- (0,2,2) -- (0,2,0) --cycle;  


               \draw[ultra thick,->,color=black!20!mycol] (0,0,0) -- (1,0,0);
    \draw[ultra thick,->,color=black!20!mycol] (0,0,0) -- (0,1,0);
    \draw[ultra thick,->,color=black!20!mycol] (0,0,0) -- (0,0,1);   
    
\path (0,0,0) node[circle, color=mycol, fill, inner sep=1.3]{};
\path (0,0,1) node[circle, color=mycol, fill, inner sep=1.3]{};
\path (0,0,2) node[circle, color=mycol, fill, inner sep=1.3]{};
\path (0,0,3) node[circle, color=mycol,fill, inner sep=1.3]{};
\path (0,0,4) node[circle, color=mycol, fill, inner sep=1.3]{};

\path (0,1,0) node[circle, color=mycol, fill, inner sep=1.3]{};
\path (0,1,1) node[circle,  color=mycol,fill, inner sep=1.3]{};
\path (0,1,2) node[circle,  color=mycol,fill, inner sep=1.3]{};
\path (0,1,3) node[circle, color=mycol, fill, inner sep=1.3]{};

\path (0,2,0) node[circle, color=mycol2, fill, inner sep=1.3]{};
\path (0,2,1) node[circle,  color=mycol2,fill, inner sep=1.3]{};
\path (0,2,2) node[circle, color=mycol2, fill, inner sep=1.3]{};

\path (0,3,0) node[circle, fill, color=mycol2, inner sep=1.3]{};
\path (0,3,1) node[circle, fill, color=mycol2, inner sep=1.3]{};

\path (0,4,0) node[circle, fill, color=mycol2, inner sep=1.3]{};

\path (1,0,0) node[circle, color=mycol, fill, inner sep=1.3]{};
\path (1,0,1) node[circle,  color=mycol,fill, inner sep=1.3]{};
\path (1,0,2) node[circle,  color=mycol,fill, inner sep=1.3]{};
\path (1,0,3) node[circle, color=mycol, fill, inner sep=1.3]{};

\path (1,1,0) node[circle, fill,  color=mycol,inner sep=1.3]{};
\path (1,1,1) node[circle, fill,  color=mycol,inner sep=1.3]{};
\path (1,1,2) node[circle, fill,  color=mycol,inner sep=1.3]{};

\path (1,2,0) node[circle, fill, color=mycol2, inner sep=1.3]{};
\path (1,2,1) node[circle, fill, color=mycol2, inner sep=1.3]{};

\path (1,3,0) node[circle, fill,color=mycol2,  inner sep=1.3]{};

\path (2,0,0) node[circle, fill,  color=mycol,inner sep=1.3]{};
\path (2,0,1) node[circle, fill,  color=mycol,inner sep=1.3]{};
\path (2,0,2) node[circle, fill,  color=mycol,inner sep=1.3]{};

\path (2,1,0) node[circle, fill,  color=mycol,inner sep=1.3]{};
\path (2,1,1) node[circle, fill,  color=mycol,inner sep=1.3]{};

\path (2,2,0) node[circle, fill, color=mycol, inner sep=1.3]{};

\path (3,0,0) node[circle, fill, inner sep=1.3]{};
\path (3,0,1) node[circle, fill, inner sep=1.3]{};

\path (3,1,0) node[circle, fill, inner sep=1.3]{};

\path (4,0,0) node[circle, fill, inner sep=1.3]{};
  \end{tikzpicture}
      \caption{Subdivision of the polytope $\Gamma_j$}
    \label{fig:polytopeK}
\end{figure}
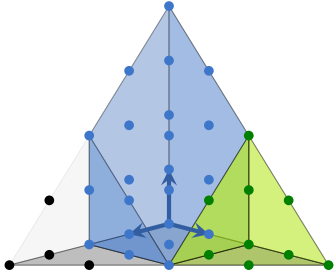
\end{center}

 \begin{prop}\label{doublehypersurface}
Let $\ell$ be a positive integer and let $d=2m$ be an even positive integer. Suppose that there exists a stably irrational double cover of $\PP^\ell_k$ branched along a smooth degree $d$ hypersurface. Then a very general hypersurface of degree $d$ in $\mathbb P^{\ell+m}_k$ is stably irrational.
\end{prop}
\begin{proof}
The Newton polytope of such a hypersurface is given by the dilated simplex
$$d\Delta_{\ell+m}=\left\{(u,v)\in \R_{\geq 0}^{\ell+m}\,|\, \sum_{i=1}^\ell u_i + v_1+\ldots+v_m \le d\right\}.
$$ The stable irrationality of $d\Delta_{\ell+m}$ follows by taking $j=0$ in Lemma \ref{subdivlemma1}.
\end{proof}

To apply this result, one needs the stable irrationality of double covers of lower dimension as input. This input typically requires additional arguments, e.g., further specializations and the decomposition of the diagonal-technique. With the current knowledge of rationality properties of double covers, Proposition \ref{doublehypersurface} unfortunately does not yield any cases that have not already been established by other methods, but we expect that it will be useful for future applications.

\section{Hypersurfaces in products of projective spaces}\label{sec:prodproj}
We can also apply our methods to the rationality problem for hypersurfaces in products of projective spaces. Like hypersurfaces in projective space, these have been extensively studied \cite{pukhlikov-paper,sobolev,pukhlikov-book, ABP,krylovokada,ABBP}. We will classify all bidegrees corresponding to stably rational/irrational hypersurfaces of dimension at most four, and also settle many open cases in higher dimension. 

The arguments here (in particular in Theorem \ref{P1P4}) utilize the full tropical machinery in Theorem \ref{theo:main}. In particular, the degenerations we use are significantly more involved than the ones considered in Section \ref{quarticfivefoldsection}, and have special fibers with large numbers of irreducible components.


\begin{prop}\label{prop:productproj}
Let $X$ be very general divisor of bidegree $(a,b)$ in $\PP^{\ell}_k\times_k \PP^m_k$, with $\ell,m\ge 2$ and $a,b\ge 1$.
\begin{enumerate}
\item \label{it:product1} If $a=1$ or $b=1$, then $X$ is rational.

\item \label{it:product2} If $\ell=2$ and $m\in \{2,3\}$, then $X$ is stably irrational if $a,b\geq 2$.


\item \label{it:product3} If $\ell=m=3$, then $X$ is stably irrational if $a>2$ or $b>2$.

\item \label{it:product4} If $\ell=2$ and $m\ge 3$, then $X$ is stably irrational for $b\ge m-1$.

\end{enumerate}
\end{prop}

\begin{proof} \eqref{it:product1} We may assume that $a=1$. Then $X$ is generically a projective bundle over $\mathbb{P}_k^m$, and therefore rational.

\eqref{it:product2} If $(a,b)=(2,2)$, then $X$ was shown to be stably irrational in \cite{HT} (for $m=2$) and \cite{HPTbideg2} (for $m=3$). The result for $a,b\geq 2$ now follows from Theorem \ref{theo:slope2}.

\eqref{it:product3} This follows from  Theorem \ref{theo:slope2} and point \eqref{it:product2}.

\eqref{it:product4} This follows  by induction on $m$ by Theorem \ref{theo:slope2}, using point \eqref{it:product2} as the base case. 
\end{proof}

\begin{prop}\label{prop:productproj2}
Let $a,b,\ell,m$ be positive integers such that $a\ge \ell+1$. If a very general hypersurface in $\PP^m_k$ of degree $b$ is stably irrational, then also a very general $(a,b)$-divisor $X$ in $\PP^\ell_k\times_k \PP^m_k$ is stably irrational.
\end{prop}

\begin{proof}
  The Newton polytope of $X$ is given by the product polytope
        $$a\Delta_\ell\times b\Delta_m=\{(u,v)\in \R_{\geq 0}^{\ell+m}\,|\,\sum_{i=1}^{\ell} u_i\le a,\ \sum_{j=1}^{m} v_j\le b\}.$$
       Starting with a unimodular subdivision of $a\Delta_\ell$ and taking the product with $b\Delta_m$, we get a regular subdivision $\mathscr P$ of $a\Delta_\ell\times b\Delta_m$ in which each polytope in $\mathscr P$ that meets the relative interior of $a\Delta_\ell\times b\Delta_m$ either has lattice width one, or is isomorphic to the polytope $b\Delta_m$, which is stably irrational by assumption. The latter possibility occurs because $a\geq \ell+1$ so that $a\Delta_\ell$ has an interior lattice point, which must be a vertex of $\mathscr P$.  Then by the formula \eqref{eq:assum} in Theorem \ref{theo:main}, we find that also $a\Delta_\ell\times b\Delta_m$ is stably irrational.
\end{proof}


Fano fibrations over $\PP^1_k$ are of special interest. There is an extensive literature on these varieties, especially from the viewpoint of birational rigidity \cite{pukhlikov-paper,sobolev, Che,pukhlikov-book,krylovokada}. The main result of this section is  the following theorem.

\begin{theo}\label{P1P4}
Let $X$ be a very general divisor of bidegree $(a,b)$ in $\PP^1_k\times_k \PP^n_k$, with $n\in \{2,3,4\}$ and $a,b\geq 1$. Then $X$ is  rational if  
$a=1$ or $b\le 2$, and stably irrational in all other cases.
\end{theo}
\begin{proof}
If $a=1$, the second projection defines a birational morphism $X\to \mathbb{P}^n_k$, so that $X$ is rational. When $b \le 2$, then $X\to \mathbb{P}^1_k$ is  generically either a projective bundle or a quadric bundle, and therefore rational by Tsen's theorem.

We now prove that $X$ is stably irrational in all other cases. The case $n=2$ is trivial, since there the canonical bundle of $X$ is semi-ample. So we may assume that $n\in \{3,4\}$. It suffices to prove that very general $(2,3)$-divisors are stably irrational; the remaining cases then follow from Theorem \ref{theo:slope2}. For $n=3$ this was proved in Theorem 1.2 in \cite{krylovokada}. So here we only need to consider the $n=4$ case.

A very general $(2,3)$-divisor $X\subset \PP^1_k\times_k \PP^4_k$ gives rise to a cubic threefold fibration over $\PP^1_k$. The second projection is generically 2-to-1, and shows that $X$ is birational to a double cover of $\PP^4_k$ branched along a determinantal sextic threefold.

Consider the Hassett--Pirutka--Tschinkel quartic \eqref{hptquartic} and its Newton polytope $\Delta$ from Example \ref{hptpolytope}. The starting observation is that $\Delta$ is contained in the Newton polytope of a general $(2,3)$ divisor:
      $$2\Delta_1\times 3\Delta_4=\{(u,v)\in \R_{\geq 0}^{1+4}\,|\,u\le 2,\ v_1+\ldots+v_4\le 3\}.$$In concrete terms, the  following bidegree $(2,3)$ polynomial \begin{dmath*}x_0^2y_0^3-2x_0x_1y_0^3+x_1^2y_0^3-2x_0^2y_0^2y_1-2x_0x_1y_0^2y_1\\+x_0^2y_0y_1^2+x_0x_1y_1y_2^2+x_0^2y_1y_3^2+x_0
       x_1y_0y_4^2
       \end{dmath*}dehomogenizes to the quartic \eqref{hptquartic} in the affine chart $D_{01}=\{x_0=y_0=1\}$. 
       
       Let $\mathscr P$ denote the regular subdivision of the polytope $2\Delta_1\times 3\Delta_4$ induced by the convex function $$f \colon \RR^5 \to \RR,\,x\mapsto \min_{z\in \Delta} \lVert{x-z}\rVert^2
$$
Using \verb"Macaulay2" or \verb"polymake" we compute the number of faces of each dimension\footnote{Macaulay2 code: \url{www.mn.uio.no/math/personer/vit/johnco/papers/subdivp1p4.m2}.}:
\begin{center}
\begin{tabular}{l|l|l|l|l|l|l|}
   $\mbox{dim }\delta$ & 0 & 1 & 2 & 3 & 4 & 5 \\ 
   \hline
   number  & 43 & 192 & 353 & 323 & 146 & 26 \\
\end{tabular}
\end{center}As $\mathscr{P}$ contains the stably irrational polytope $\Delta$ (which has dimension 5), it suffices by Theorem \ref{theo:main} to show that any polytope in $\mathscr{P}$ of even dimension is either stably rational or contained in the boundary of $2\Delta_1\times 3\Delta_4$. Going through the faces of dimension 2 and 4 reveals that any face $\delta$ of even dimension either has lattice width one or is one of the following two polytopes:

{    \begin{enumerate}[(a)]
 \item \label{it:casea} Conv\{(0, 3, 0, 0, 0), (1, 3, 0, 0, 0), (2, 2, 1, 0, 0), (1, 1, 2, 0,
      0), (2, 1, 2, 0, 0), (0, 1, 0, 2, 0), (1, 1, 0, 2, 0), (1, 1, 0, 0, 2),
      (2, 1, 0, 0, 2)\}
 \item \label{it:caseb} Conv\{(2, 0, 0, 0, 0), (0, 2, 0, 0, 0), (1, 1, 2, 0, 0), (0, 1, 0, 2,
      0), (1, 0, 0, 0, 2)\}
     \end{enumerate}}

    
In each of these cases the corresponding hypersurface $Y$ is rational as well: in the case \eqref{it:casea}, translating the polytope by the vector $(0,-1,0,0,0)$ shows that $Y$ is a quadric bundle over $\PP^1_k$, hence rational. In the case \eqref{it:caseb}, the hypersurface $Y$ is defined by an affine equation of the form 
$$\alpha z_0 z_1 z_2^2+\beta z_0 z_1 z_2+\gamma z_1 z_3^2+\delta
       z_0 z_4^2+\varepsilon z_0^2+\zeta z_0 z_1+\eta z_1^2=0.$$
The projection onto $\Spec k[z_2,z_3,z_4]$ defines a conic bundle with a section (given by $z_0=z_1=0$), and hence $Y$ is rational. 
\end{proof}
Proposition \ref{prop:productproj}  and Theorem \ref{P1P4} together  completely settle the classification of stably irrational hypersurfaces in $\PP^\ell_k \times_k \PP^m_k$ of dimension at most four. Combined with  Proposition \ref{prop:productproj2}, they also cover a wide range of cases in higher dimensions.

\section{Complete intersections}\label{sec:ci}
\subsection{Intersections of a quadric and a cubic in $\mathbb{P}^6$}\label{ss:quadcub}
 We now turn our attention to complete intersections. 
Table \ref{tab:fano4} collects what is known about stable rationality of very general Fano complete intersection fourfolds (excluding the elementary cases of linear spaces and quadrics in $\mathbb{P}^5_k$, and the classical case of intersections of two quadrics in $\mathbb{P}^6_k$; all of these cases are rational). The first column displays the multidegrees, the second column indicates whether a very general complete intersection with these multidegrees is stably rational, and the third column contains the references to the original results.

\begin{table}[!htbp]
\begin{tabular}{| l | l | l |}
\hline
  \mbox{multidegree} & stably rational? & reference \\
  \hline 
  4\mbox{ or }5 & No & Totaro \cite{totaro} \\
   (2,4) & No & Chatzistamatiou-Levine \cite[6.1]{CL}\\
   (3,3) & No & Chatzistamatiou-Levine \cite[6.1]{CL}\\
  (2,2,2) & No & Hassett-Pirutka-Tschinkel \cite{HPT3quad} \\
  (2,2,3)  & No & Chatzistamatiou-Levine \cite[6.1]{CL} \\
  (2,2,2,2)  & No & Chatzistamatiou-Levine \cite[6.1]{CL} \\
  \hline
\end{tabular}
\vspace{3pt}
\caption{Rationality properties of very general Fano complete intersection fourfolds}
\label{tab:fano4}
\end{table}
The only open cases are the cubic fourfold and the $(2,3)$ complete intersection in $\mathbb{P}^6_k$. We will settle the latter case by proving that a very general intersection of a quadric and a cubic in $\mathbb{P}^6_k$ is not stably rational.
 Our proof is again based on degeneration techniques. Rather than using tropical tools, we will construct a suitable semi-stable degeneration by hand and then apply the motivic obstruction provided by Corollary \ref{coro:obstruction}.

\begin{theo}\label{theo:quadcub}
A very general intersection of a quadric and a cubic in $\mathbb{P}^6_k$ is stably irrational. 
\end{theo}
In fact, we will show that a very general complete intersection of a quadric and a cubic containing a $2$-plane in $\PP^6_k$ is stably irrational.
\begin{proof}
By Corollary \ref{coro:verygen}, it suffices to find an algebraically closed field $F$ of characteristic zero and a quadric and cubic in $\mathbb{P}^{n+2}_F$ whose intersection is a stably irrational smooth $n$-fold. We will construct such an example over $F=K$.

 We first introduce some notation. For every non-negative integer $\ell$ and every tuple $u$ in $\N^{\{0,\ldots,\ell\}}$, we denote by $|u|$ the sum of the coordinates of $u$, and we use the
 multi-index notation $x^{u}$ to denote the monomial $x_0^{u_0}\cdots x_\ell^{u_\ell}$.
  For every $i$ in $\N^{\{0,1,2,3\}}$ and every $j$ in $\N^{\{4,5,6\}}$, we will denote by $ij$
 the tuple in $\N^{\{0,\ldots,6\}}$ obtained by concatenating $i$ and $j$.
We set
$$I=\{(i,j)\in \N^{\{0,1,2,3\}}\times \N^{\{4,5,6\}}\,|\ |ij|=3\}.$$

 Let $\cX$ be the projective flat $k\llbr t\rrbr$-scheme defined by
$$\cX=\mathrm{Proj}\,k\llbr t\rrbr[x_0,\ldots,x_6]/(x_0x_1-x_2x_3+t^2(x_4^2+x_5^2+x_6^2)).$$
Let $\cY$ be the closed subscheme of $\cX$ defined by the cubic equation
\begin{equation}\label{eq:cubic}
\sum_{(0,j)\in I}ta_{0j}x^{0j}+\sum_{(i,j)\in I,\,i\neq 0}a_{ij}x^{ij}=0
\end{equation}
for a very general choice of coefficients $a_{ij}$ in $k$.
    We will prove that $\cY_K$ is not stably rational. 

The special fiber $\cX_k$ is the toric hypersurface in $\mathbb{P}^6_k$ defined by $x_0x_1-x_2x_3=0$.
   The scheme $\cX$ is singular along the plane $P$ in $\cX_k$ defined by $t=x_0=x_1=x_2=x_3=0$. The closed subscheme $\cY_k$ is the intersection of $\cX_k$ with a very general cubic hypersurface in $\mathbb{P}^6_k$ containing the plane $P$.
 Let $\pi\colon \cY'\to \cY$ be the blow-up of $\cY$ along $P$.

  We claim that $\cY'$ is regular, and that its special fiber $\cY'_k$ is a reduced divisor with strict normal crossings. Moreover, $\cY'_k$ has two irreducible components $E_1$ and $E_2$, where $E_1$ is the strict transform of $\cY_k$. Finally, we also claim that $E_1$ is stably irrational and that $E_1\cap E_2$ is rational. Then by applying Corollary \ref{coro:obstruction} to the strictly semi-stable model $\cY'\times_{k\llbr t \rrbr}R$ of $\cY_K$, we see that $\cY_K$ is not stably rational.

  So let us prove our claims. The plane $P$ is the base locus of the linear system of cubic forms  $$\sum_{(i,j)\in I,\,i\neq 0}a_{ij}x^{ij}$$ on the toric variety $\cX_k$. By Bertini's theorem, the $k$-scheme $\cY_k$ is integral, and smooth away from $P$, for a general choice of coefficients $a_{ij}$ in $k$.  Therefore, we only need to investigate the structure of $\cY'$ above a neighbourhood of $P$.
 On $P$ at least one of the homogeneous coordinates $x_4$, $x_5$ and $x_6$ is non-zero; by the symmetry of our equations, it will be sufficient to consider the affine chart $\cU$ of $\cY$ where $x_6\neq 0$.
  To analyze the geometry of the blow-up $\cU'\to \cU$ at $P\cap \cU$, it suffices to compute the $t$-chart and the $x_0$-chart, because our equations are also symmetric in the variables $x_0,\ldots,x_3$.

Let $\cV$ be the closed subscheme of
$\mathrm{Proj}\, k\llbr t\rrbr [w_0,\ldots,w_6]$ defined by the equations
$$\left\{\begin{array}{rcl}
 w_0w_1-w_2w_3+w_4^2+w^2_5+w_6^2&=&0,
   \\[1.5ex] \displaystyle \sum_{(0,j)\in I}a_{0j}w^{0j}+ \sum_{(i,j)\in I,\,i\neq 0}t^{|i|-1}a_{ij}w^{ij}&=&0.
   \end{array}\right.$$
   The $t$-chart $\cU^{(t)}$ of the blow-up $\cU'$ (by which we mean the chart where the inverse image of $P$ is defined by $t=0$) is the affine chart of $\cV$ where $w_6\neq 0$.
The first equation defines a smooth quadric, so it follows from Bertini's theorem that the special fiber  $\cV_k$
 is integral, and smooth away from the base locus $w_4=w_5=w_6=0$ of the linear system of cubics
$$\sum_{(i,j)\in I,\,|i|\leq 1}a_{ij}w^{ij} .$$
Thus the open subscheme  $\cU^{(t)}_k$ of $\cV_k$ defined by $w_6\neq 0$ is smooth and connected.

  The $x_0$-chart $\cU^{(x_0)}$ of the blow-up $\cU'$ is the closed subscheme of
   $\Spec k\llbr t\rrbr [s,v_0,\ldots,v_6]$ defined by the equations
$$\left\{\begin{array}{rcl}
v_6&=&1,
\\[1.5ex] sv_0&=&t,
  \\[1.5ex] v_1-v_2v_3+s^2(v_4^2+v^2_5+1)&=&0,
   \\[1.5ex]\displaystyle s(\sum_{(0,j)\in I}a_{0j}v^{0j})+ \sum_{(i,j)\in I,\,i\neq 0}a_{ij}v_0^{|i|-1}v^{\widetilde{\imath}j}&=&0,
   \end{array}\right.$$
where we set $\widetilde{\imath}=(0,i_1,i_2,i_3)$.
 The special fiber $\cU^{(x_0)}_k$ is reduced and consists of two smooth irreducible components,
 $D_1$ and $D_2$. The component $D_1$ is the closed subscheme of
   $\Spec k[v_0,\ldots,v_6]$ cut out by the equations
$$\left\{\begin{array}{rcl}
v_6&=&1,
  \\[1.5ex] v_1-v_2v_3&=&0,
   \\[1.5ex] \displaystyle \sum_{(i,j)\in I,\,i\neq 0}a_{ij}v_0^{|i|-1}v^{\widetilde{\imath}j}&=&0.
   \end{array}\right.$$
   The closure of $D_1$ in $\cY'_k$ is the strict transform of $\cY_k$. After the substitution  $v_1=v_2v_3$, the third equation becomes a polynomial in $k[v_0,v_2,\ldots,v_6]$
of the form
$$\sum_{p,q\in \{0,4,5,6\},\,p \leq q}g_{p,q}(v_2,v_3)v_pv_q$$  where $g_{p,q}(v_2,v_3)$ is a very general polynomial of bidegree $(c+1,c+1)$ with $c$ the number of occurrences of the coordinate $0$ in the couple $(p,q)$.

  It follows from Corollary 11 in \cite{schreieder-conic} that $D_1$ is not stably rational; in the notations of \cite{schreieder-conic}, the scheme $D_1$ is an open subscheme of a very general quadric surface bundle of lexicographically ordered type $((1,1),(1,1),(1,1),(3,3))$.

  The component $D_2$ is the closed subscheme of
   $\Spec k[s,v_1,\ldots,v_6]$ cut out by the equations
$$\left\{\begin{array}{rcl}
v_6&=&1,
\\[1.5ex] v_1-v_2v_3+s^2(v_4^2+v^2_5+1)&=&0,
\\[1.5ex] \displaystyle s(\sum_{(0,j)\in I}a_{0j}v^{0j})+ \sum_{(i,j)\in I,\,|i|=1}a_{ij}v^{\widetilde{\imath}j}&=&0.
 \end{array}\right.$$

The schematic intersection $D_1\cap D_2$ is
the closed subscheme of
   $\Spec k[v_1,\ldots,v_6]$ cut out by the equations
$$\left\{\begin{array}{rcl}
v_6&=&1,
\\[1.5ex] v_1-v_2v_3&=&0,
\\[1.5ex] \displaystyle \sum_{(i,j)\in I,\,|i|=1}a_{ij}v^{\widetilde{\imath}j}&=&0.
 \end{array}\right.$$
This is a smooth hypersurface in $\Spec k[v_2,v_3,v_4,v_5]$.
 It is rational, because it is defined by an irreducible polynomial that has degree $1$ in the variable $v_2$.

Combining these calculations, we find that $\cY'_k$ is reduced, and consists of two irreducible components $E_1=\overline{D_1}$ and $E_2=\overline{D_2}$. These components are smooth over $k$, the component $E_1$ is stably irrational, and the intersection $E_1\cap E_2=\overline{D_1\cap D_2}$ is a smooth rational $k$-variety. Our local calculations also show that $E_1$ and $E_2$ are Cartier along $E_1\cap E_2$. It follows that $\cY'$ is regular, and that $\cY'_k$ is a reduced divisor with strict normal crossings.
\end{proof}

\subsection{Intersections of two cubics in $\mathbb{P}^7$}
We can use the result in the previous section to settle another open problem about stable rationality of complete intersections: we will prove that a very general intersection of two cubics in $\mathbb{P}^7_k$ is stably irrational.

\begin{theo}\label{theo:33}
Let $n$ be an integer such that $n\geq 3$. Assume that a very general intersection of a quadric and a cubic in $\mathbb{P}^{n+1}_k$ is not stably rational.  
 Then a very general intersection of two cubics in $\mathbb{P}^{n+2}_k$ is not stably rational.
\end{theo}
\begin{proof}
Let $q_1$ and $q_2$ be quadratic forms in $k[x_0,\ldots,x_{n+2}]$, and let   
 $c_1$ and $c_2$ be cubic forms in $k[x_0,\ldots,x_{n+2}]$. We assume that the tuple $(q_1,q_2,c_1,c_2)$ is very general.  Let $\cX$ be the closed subscheme of 
$$\mathbb{P}^{n+2}_{R}=\mathrm{Proj}\,R[x_0,\ldots,x_{n+2}]$$ defined by the equations $$c_1=tc_2-x_{n+2}q_2=0.$$
Set $X=\cX\times_{R}K$; this is a smooth complete intersection of two cubics in $\mathbb{P}^{n+2}_K$. 
  The scheme $\cX$ is strictly toroidal, because, Zariski-locally, it admits a smooth morphism to a scheme as in Example \ref{exam:standardmod}. Thus, we can use Theorem \ref{theo:vol} to compute the stable birational volume of $X$. The special fiber $\cX_k$ consists of two irreducible components, $E_1$ and $E_2$, where $E_1$ is the smooth cubic in $\mathbb{P}^{n+1}_k$ defined by $c_1(x_0,\ldots,x_{n+1},0)=0$, and $E_2$ is the closed subscheme of 
 $\mathbb{P}^{n+2}_k$ defined by $c_1=q_2=0$. The intersection $E_{12}=E_1\cap E_2$ is a very general intersection of a quadric and a cubic in $\mathbb{P}^{n+1}_k$; thus it is stably irrational, by the assumption of the theorem. We have 
 $$\Volsb(\sbir{X})=\sbir{E_1}+\sbir{E_2}-\sbir{E_{12}}$$ in $\Z[\SB_k]$. We will prove that this element is different from $\sbir{\Spec k}$; then the result follows from Corollary  \ref{coro:obstruction}. 
 
 It suffices to show that $\sbir{E_1}$ and $\sbir{E_2}$ are both different from $\sbir{E_{12}}$. 
  We first consider the case of $E_1$. We consider the closed subscheme $\cY$ of $\mathbb{P}^{n+2}_{k[t]}$ defined by 
 $$tc_1-x_{n+1}q_2=x_{n+2}=0.$$
 For every value of $t$ in $k^{\times}$, the corresponding member of the family $\cY$
 contains the subscheme $E_{12}$ of $\mathbb{P}^{n+2}_k$.  
The scheme $\cY\times_{k[t]}R$ is strictly toroidal, and its special fiber is a union of two smooth and proper rational varieties intersecting along a smooth rational subvariety; thus it follows from Theorem \ref{theo:vol} that 
$$\Volsb(\sbir{\cY\times_{k[t]}K})=\sbir{\Spec k}\neq \sbir{E_{12}}$$ in $\Z[\SB_k]$.  This implies that  $\cY\times_{k[t]}K$ is not stably birational to $E_{12}\times_k K$, so that a very general cubic in $\mathbb{P}_k^{n+1}$ that contains $E_{12}$ is not stably birational to $E_{12}$, by Corollary \ref{coro:verygen}. 

 The case of $E_2$ is quite similar. Let $\cZ$ be the closed subscheme of $\mathbb{P}^{n+2}_{k[t]}$ defined by 
 $$tc_1-x_{n+2}q_1=q_2=0.$$ For all values of $t$ in $k^{\times}$, the corresponding fiber of the family $\cZ$ contains $E_{12}$. The fiber $\cZ_0$ of $\cZ$ over $0$ is a union of two smooth and proper rational varieties intersecting along a smooth rational subvariety (each of these varieties is a quadric or an intersection of two quadrics). The same argument as above shows that $E_2$ is not stably birational to $E_{12}$.  
  \end{proof}
\begin{coro}\label{coro:33}
A very general intersection of two cubics in $\mathbb{P}^7_k$ is not stably rational.
\end{coro}
\begin{proof}
By Theorem \ref{theo:quadcub}, a very general intersection of a quadric and a cubic in $\mathbb{P}^6_k$ is not stably rational. Thus the result follows from Theorem \ref{theo:33}.
\end{proof}

\subsection{Bounds for complete intersections}\label{ss:ci}
As another application, we will provide a significant improvement of the degree bounds for stably irrational complete intersections, bringing them up to the same level as Schreieder's results for hypersurfaces in \cite{schreieder}. 
 We use Schreieder's bounds for hypersurfaces as an essential ingredient of the proof. We first establish a generalization of Corollary \ref{coro:hyperplane}.

\begin{prop}\label{prop:linear}
Let $d$ and $n$ be positive integers with $n\geq 2$ and let $X$ be a very general hypersurface of degree $d$ in $\mathbb{P}^{n+1}_k=\mathrm{Proj}\,k[z_0,\ldots,z_{n+1}]$.  
 For every strict subset $I$ of $\{0,\ldots,n+1\}$, we denote by $L_I$ the linear subspace of $\mathbb{P}^{n+1}_k$ defined by $z_i=0$ for all $i$ in $I$. In particular, $L_{\emptyset}=\mathbb{P}^{n+1}_k$. 
 
 Let $J$ be a subset of $\{0,\ldots,n+1\}$ of cardinality at most $n-1$ and assume that $X\cap L_J$ is not stably rational. Then for every strict subset $J'$ of $J$, the scheme $X\cap L_{J'}$ is not stably birational to $X\cap L_J$. 
\end{prop} 
\begin{proof}
Replacing $\mathbb{P}^{n+1}_k$ by $L_{J'}$, we can reduce to the case where $J'$ is empty and $J$ is non-empty. Then we must show that $X$ is not stably birational to $X\cap L_J$. 
Set $$\Delta=\{(u_0,\ldots,u_{n+1})\in \R^{n+2}_{\geq 0}\,|\,u_0+\ldots+u_{n+1}=d\}$$
and let $\delta$ be the face of $\Delta$ defined by $u_j=0$ for all $j\in J$. 
 Let $g$ be a very general homogeneous polynomial of degree $d$ in $k[x_0,\ldots,x_{n+1}]$. Then $g$ has Newton polytope $\Delta$, and it suffices to show that $Z^o(g)$ is not stably birational to $Z^o(g_{\delta})$. By our assumption that $X\cap L_J$ is not stably rational, we know that $\delta$ is  stably irrational. If $\Delta$ is stably rational then the result is clear. Otherwise, it follows from Theorem \ref{theo:variation}\eqref{it:varyface}: after a permutation of the coordinates, we may assume  that $J$ contains the index $0$, and then we can use the same subdivision $\mathscr{P}$ as in the proof of Corollary \ref{coro:hyperplane}.
\end{proof}

\begin{theo}\label{theo:ci}
Let $n$ and $r$ be positive integers. Let $d_1,\ldots,d_r$ be positive integers such that $d_r\geq d_i$ for all $i$. Assume that   
$$n+r\geq \sum_{i=1}^{r-1} d_i +2$$
and that 
 there exists a stably irrational smooth hypersurface of degree  $d_r$ in $\mathbb P^{n+r-\sum_{i=1}^{r-1}d_i}_k$.
Then a very general complete intersection in $\mathbb{P}^{n+r}_k$ of multidegree $(d_1,\ldots,d_r)$ is not stably rational. 
\end{theo}
\begin{proof}
By Corollary \ref{coro:verygen}, it suffices to construct a stably irrational smooth complete intersection of multidegree $(d_1,\ldots,d_r)$ in $\mathbb{P}^{n+r}_F$, for some algebraically closed field $F$ of characteristic zero. We will take for $F$ the field $K$ of Puiseux series over $k$. If we set $d_{0}=n+r+1-\sum_{i=1}^{r-1} d_i$, then, together with the assumption in the statement, Corollary \ref{coro:verygen} also implies that 
 a very general  hypersurface of degree  $d_r$ in $\mathbb P^{d_0-1}_k$ is stably irrational.

We denote the homogeneous coordinates on $\mathbb{P}^{n+r}_k$ by $z_{ij}$ where $i$ ranges from $0$ to $r-1$ and $j$ ranges from $1$ to $d_i$. We denote by $$A=k[z_{ij}\,|\,i=0,\ldots,r-1;\,j=1,\ldots,d_i]$$
the homogeneous coordinate ring of $\mathbb{P}^{n+r}_k$.

 Let $(F_1,\ldots,F_r)$ be a very general tuple of homogeneous polynomials in $A$ of multidegree $(d_1,\ldots,d_r)$. 
 Let $\cX$ be the closed subscheme of $\mathbb{P}^{n+r}_{R}=\mathrm{Proj}\,(A\otimes_k R)$ defined by 
 $$ tF_1 - \prod_{j=1}^{d_1}z_{1j}=\ldots=tF_{r-1}-\prod_{j=1}^{d_{r-1}}z_{(r-1)j}=0.$$  
  Let $\cY$ be the closed subscheme of $\cX$ defined by $F_r=0$. Then  
  $\cY_K=\cY\times_{R}K$ is a smooth complete intersection of multidegree $(d_1,\ldots,d_r)$ in $\mathbb{P}^{n+r}_K$. We will show that $\cY_K$ is stably irrational by means of the motivic obstruction from Corollary \ref{coro:obstruction}.

 The scheme $\cY$ is strictly toroidal, so that we can compute the stable birational volume of $\cY_K$ by means of the formula in Theorem \ref{theo:vol}.
    Let $S$ be the set of couples $(i,j)$ where $i$ ranges from $1$ to $r-1$ and $j$ ranges from $1$ to $d_i$. We say that a subset $T$ of $S$ is {\em admissible}
    if, for every $i$ in $\{1,\ldots,r-1\}$, there exists an element $j$ in $\{1,\ldots,d_i\}$ such that $(i,j)$ belongs to $T$. For every admissible subset $T$ of $S$, we denote by $L_T$ the linear subspace of $\mathbb{P}^{n+r}_k$ defined by $z_{ij}=0$ for all $(i,j)$ in $T$.
     Then the irreducible components of $\cY_k$ are the schemes $Z(F_r)\cap L_T$ where $T$ runs through the minimal admissible subsets of $S$. The intersection of all these irreducible components is $Z(F_r)\cap L_S$. This is a very general degree $d_r$ hypersurface in 
     $$L_S=\mathrm{Proj}\,k[z_{01},\ldots,z_{0d_0}]\cong \mathbb{P}_k^{d_0-1},$$ 
     and, therefore, stably irrational.
     
   Theorem \ref{theo:vol} implies that  
\begin{equation}\label{eq:ci}
\Volsb(\sbir{\cY_K})=\sum_{T\subset S\,\mathrm{adm.}}(-1)^{|T|-r+1}\sbir{Z(F_r)\cap L_T}
\end{equation} 
in $\Z[\SB_k]$.
   It follows from Proposition \ref{prop:linear} that 
 $Z(F_r)\cap L_T$ is not stably birational to $Z(F_r)\cap L_S$, for all admissible strict subsets $T$ of $S$. This means that the term $\sbir{Z(F_r)\cap L_S}$ does not cancel out in the formula \eqref{eq:ci}, so that $\cY_K$ is stably irrational by Corollary \ref{coro:obstruction}.
\end{proof}

\begin{coro}\label{coro:ci}
Let $n$ and $r$ be positive integers. Let $d_1,\ldots,d_r$ be positive integers such that $d_r\geq 4$ and $d_r\geq d_i$ for all $i$.
 Assume that  
$$
\sum_{i=1}^{r-1} d_i +2 \le n+r\leq 2^{d_r-2}+\sum_{i=1}^r d_i -3.
$$
Then a very general complete intersection in $\mathbb{P}^{n+r}_k$ of multidegree $(d_1,\ldots,d_r)$ is not stably rational.
\end{coro}
\begin{proof}
Our assumptions imply that a very general hypersurface of degree  $d_r$ in $\mathbb P^{n+r-\sum_{i=1}^{r-1}d_i}_k$ is stably irrational,  by Theorem 1.1 in \cite{schreieder}.
\end{proof}

To conclude, we prove the following variant of the above theorem by applying Corollary \ref{coro:obstruction} in a slightly different way.
This in particular extends  the results for intersections of quadrics from \cite{HPT3quad} and \cite{CL} (see Corollary \ref{coro:quadrics}).

\begin{theo}\label{theo:ci2}
Let $n$ be a positive integer and let $d$ be an integer which is either $\le 34$ or a prime power. 
Let $d_1,\ldots,d_r$ be positive integers and assume that a very general complete intersection
of multidegree $(d_1,\ldots,d_r)$ in $\mathbb{P}^{n+r}_k$ is stably irrational. Then a very general
complete intersection of multidegree $(d,d_1,\ldots,d_r)$ in $\mathbb{P}^{n+r+d}_k$ is also stably irrational. 
\end{theo}
\begin{proof}
Consider the polynomial ring $A=k[z_0,\ldots,z_{n+r},w_1,\ldots,w_d]$. The symmetric group $G=S_d$ acts on $A$ by permuting the variables $w_1,\ldots,w_d$  (leaving the $z_i$ unchanged). Let $(F_0,\ldots,F_{r})$ be an $(r+1)$-tuple of homogeneous polynomials in $A$ of degrees $d,d_1,\ldots,d_r$ respectively. We assume that this tuple is very general subject to the condition that the forms $F_1,\ldots,F_{r}$ are invariant under the action of $G$.
 Let $\cY$ be the closed subscheme of $$\mathbb{P}^{n+r+d}_{R}=\mathrm{Proj}\,(A\otimes_k R)$$ defined by the equations 
$$F_1=\ldots=F_{r}=tF_{0} - w_1\cdots w_d = 0.$$
 The $R$-scheme $\cY$ is strictly toroidal, by Example \ref{exam:standardmod}.
 Thus we can apply Theorem \ref{theo:vol} to compute the stable birational volume   
 $\Volsb(\sbir{\cY_K})$ as an alternating sum of the classes of the strata of $\cY_k$ in $\Z[\SB_k]$. 
 
 The special fiber $\cY_k$ consists of $d$ irreducible components $Y_1,\ldots,Y_d$, which are permuted by the action of $G$ on $\mathbb{P}^{n+r+d}_k$. 
  For each non-empty subset $J$ of $\{1,\ldots, d\}$, let $Y_J=\cap_{j\in J}Y_j$. This scheme is isomorphic to a complete intersection in $\mathbb{P}^{n+r+d-|J|}_k$ of multidegree $(d_1,\ldots,d_r)$. For sets $J$ of fixed cardinality, the schemes $Y_J$ are permuted by the action of $G$, so in particular they are all isomorphic. The smallest intersection  $Y_{\{1,\ldots,d\}}$ is  a very general complete intersection of multidegree $(d_1,\ldots,d_r)$ in $\mathbb{P}^{n+r}_k$, which is stably irrational by assumption. From this we get
 $$\Volsb(\sbir{\cY_K})={d \choose 1}\sbir{Y_1}-{d \choose 2}\sbir{Y_{\{1,2\}}}+\ldots+(-1)^{d+1}\sbir{Y_{\{1,\ldots,d\}}}.$$If $d$ is a prime power, say $d=p^\nu$, then each of the $d-1$ first binomial coefficients appearing in the sum is divisible by $p$. This implies that the above expression is different from $\sbir{\Spec\,k}$ in $\Z[\SB_k]$. the same conclusion holds if $d\le 34$ by a direct computation. Thus $\cY_K$ is not stably rational, by Corollary \ref{coro:obstruction}.
\end{proof}

%
%
%
\begin{coro}\label{coro:quadrics}
Let $n$ be an integer such that $n\geq 3$. Let $r$ be an integer such that $r\geq 3$ and $r \geq n-1$. 
 Then a very general complete intersection of $r$ quadrics in $\mathbb{P}^{n+r}_k$ is stably irrational. 
\end{coro}
\begin{proof}
  We only need to consider the cases where $n=r$ or $n=r+1$, since, for $n<r$, a complete intersection of $r$ quadrics in $\mathbb{P}^{n+r}_k$ has non-negative Kodaira dimension.
  For $r=3$, the $n=3$ case is proved in \cite[\S4.4]{HT} and the $n=4$ case is proved in \cite{HPT3quad}. For $r\geq 4$ the result follows from Theorem \ref{theo:ci2}, by induction on $r$.
 \end{proof}
 
 \begin{coro}\label{coro:ci2}
 Let $n$ and $d$ be positive integers. Assume that a very general degree $d$ hypersurface in $\mathbb{P}^{n+1}_k$ is stably irrational. Then very general complete intersections of the following multidegrees are stably irrational:
\begin{itemize} 
\item  multidegree $(2,d)$ in $\mathbb{P}^{n+3}_k$; 
\item multidegree $(3,d)$ in $\mathbb{P}^{n+4}_k$;
\item multidegree $(2,2,d)$ in $\mathbb{P}^{n+5}_k$;
\item  multidegree $(2,3,d)$ in $\mathbb{P}^{n+6}_k$.
\end{itemize} 
 \end{coro}
 \begin{proof}
 These are special cases of Theorem \ref{theo:ci2}.
 \end{proof}
 
 Theorems \ref{theo:ci} and \ref{theo:ci2} settle many new cases for the rationality question for very general complete intersections, especially in high dimensions, and at the same time cover a large portion of the previously known cases in a uniform way. To give a sample in low dimension, 
 the combined results in this paper show that very general complete intersection fivefolds of the following multidegrees are stably irrational: 
$$\begin{array}{l}
(4), (5), (6), (2,4), (2,5), (3,3), (3,4), (2,2,3), (2,2,4), (2,3,3), 
\\ (2,2,2,2), (2,2,2,3), (2,2,2,2,2).
\end{array}$$  
To the best of our knowledge, the cases $(4)$, $(3,3)$, $(2,2,3)$ and $(2,2,2,2)$ are new.  
 The cases $(5)$ and $(6)$ lie in Schreieder's range \cite{schreieder}, and the remaining ones are covered by Theorem 6.1 in \cite{CL}. Since a smooth complete intersection of three quadrics in $\mathbb{P}^8_k$ is rational by \cite{tjurin}, the only remaining open cases in dimension $5$ are the cubic fivefold and the intersection of a quadric and a cubic in $\mathbb{P}^7_k$. 
 
 In dimension six, our results imply that very general complete intersections of the following multidegrees are stably irrational:   
$$\begin{array}{l}
(5), (6), (7), (2,4), (2,5), (2,6), (3,4), (3,5), (4,4), 
  (2,2,4), (2,2,5), (2,3,3), (2,3,4), 
  \\ (3,3,3),
 (2,2,2,3), (2,2,2,4), (2,2,3,3), (2,2,2,2,2), (2,2,2,2,3), (2,2,2,2,2,2)
\end{array}$$    
Here the cases not covered by \cite{schreieder} and \cite{CL} are    
$$(2,4), (2,5), (2,3,3), (2,2,2,3), (2,2,2,2,2), (2,2,2,2,2,2).$$

\end{document}